\documentclass[12pt]{amsart} %
\usepackage[top=1in, bottom=1in, left=1in, right=1in]{geometry} 
\usepackage{float}
\usepackage{placeins}
\newcommand{\de}{\partial}
\newcommand{\g}{\mathfrak{S}}
\usepackage{tikz}
\usepackage{hyperref}
\usepackage{breqn}
\usepackage{titletoc}
\usepackage{graphicx}
\usepackage{color} 
\usepackage{amsmath, amsfonts, amssymb, amsxtra, amsthm}
\usepackage{mathrsfs}
\usepackage{dsfont}
\usepackage{verbatim}
\usepackage{comment}
\usepackage{enumitem}
\DeclareMathOperator{\bwt}{bwt}
\newcommand{\figscale}{0.6}
\newcommand{\inlinescale}{0.37}
\usepackage{indentfirst}
\numberwithin{equation}{section}

\usepackage{amssymb} 
\usepackage{enumerate} 
\usepackage{hyperref} 
\hypersetup{colorlinks = true,
    linkcolor={red},
    citecolor={blue},
    urlcolor={magenta}} 
\usepackage{tikz} 

\newcommand{\Z}{\mathbb Z}

\DeclareMathOperator{\wt}{wt}
\DeclareMathOperator{\QBPD}{QBPD}

\DeclareMathOperator{\BPD}{BPD}
\DeclareMathOperator{\id}{id}

\theoremstyle{plain} 
\newtheorem{theorem}{Theorem}[section]
\newtheorem{proposition}[theorem]{Proposition} 
 
\newtheorem{lemma}[theorem]{Lemma}
\newtheorem{corollary}[theorem]{Corollary}
\newtheorem{setup}[theorem]{Setup}
\newtheorem{example}[theorem]{Example}

\theoremstyle{definition} 
\newtheorem{definition}[theorem]{Definition} 
\newtheorem{remark}[theorem]{Remark}

\title{Quantum bumpless pipe dreams}
\author{Tuong Le}
\author{Shuge Ouyang}
\author{Leo Tao}
\author{Joseph Restivo}
\author{Angelina Zhang}

\address[Le]{
Dept.\ of Mathematics\\
   University of Michigan\\
   Ann Arbor, MI
}
\email{tuongle@umich.edu}

\address[Ouyang]{
Dept.\ of Mathematics\\
   University of Michigan\\
   Ann Arbor, MI
}
\email{shugeoy@umich.edu}

\address[Tao]{
Dept.\ of Mathematics\\
   University of Michigan\\
   Ann Arbor, MI
}
\email{leotao@umich.edu}

\address[Restivo]{
Dept.\ of Mathematics\\
   University of Michigan\\
   Ann Arbor, MI
}
\email{restivo@umich.edu}

\address[Zhang]{
Dept.\ of Mathematics\\
   University of Michigan\\
   Ann Arbor, MI
}
\email{angzh@umich.edu}

\date{\today}

\begin{document}

\begin{abstract}
    
    Schubert polynomials are polynomial representatives of Schubert classes in the cohomology of the complete flag variety and have a combinatorial formulation in terms of bumpless pipe dreams. Quantum double Schubert polynomials are polynomial representatives of Schubert classes in the torus-equivariant quantum cohomology of the complete flag variety, but no analogous combinatorial formulation had been discovered. We introduce a generalization of the bumpless pipe dreams called quantum bumpless pipe dreams, giving a novel combinatorial formula for quantum double Schubert polynomials as a sum of binomial weights of quantum bumpless pipe dreams. We give a bijective proof for this formula by showing that the sum of binomial weights satisfies a defining transition equation.
\end{abstract}

\maketitle

\section{Introduction}
Schubert polynomials were introduced by Lascoux and Schützenberger \cite{Lascoux82} and they represent cohomology classes called Schubert classes of the complete flag variety. The original definition was algebraic and relied on divided difference operators; however, multiple combinatorial formulas for the monomial expansion of Schubert polynomials were found \cite{assaf, magyar, FOMIN1996123, bergeron, Lam_2021}. Two such examples are Schubert polynomials as weight-generating functions of pipe dreams (originally called RC-graphs \cite{FOMIN1996123, bergeron}), or bumpless pipe dreams \cite{Lam_2021}. 
For example, for $\BPD(w)$, the set of (reduced) bumpless pipe dreams of a permutation $w$, one has \begin{equation}
\g_w(x, y) = \sum_{P\in \BPD(w)}\bwt(P) \,,\end{equation}
where $\bwt(P)$ is a product of binomials $(x_i-y_j)$ associated to $P\in \BPD(w)$ \cite{Lam_2021}.
These pipe dream and bumpless pipe dream formulations  generalize to some generalizations of Schubert polynomials, such as double Schubert polynomials and double Grothendieck polynomials \cite{KNUTSON2004161,WEIGANDT2021105470}.

Motivated by ideas that stem from string theory, mathematicians defined quantum cohomology rings \cite{vafa, witten}. See, e.g.,~\cite{Givental97} for more on the history of quantum cohomology. In the quantum cohomology of the complete flag variety, the Schubert classes correspond to quantum Schubert polynomials, another generalization of Schubert polynomials \cite{quantum}. Quantum double Schubert polynomials, which generalize both quantum Schubert polynomials and double Schubert polynomials, were defined in \cite{KIRILLOV2000191,cf}. Like Schubert polynomials, there is a quantum double Schubert polynomial for each permutation $w$ of $\{1, 2, \dots, n\}$, denoted $\g_w^q(x,y)$, lying in $\Z[x_1, \dots, x_n, y_1, \dots, y_n, q_1, \dots, q_{n-1}]$. There is no known combinatorial formulation for the monomial expansion of quantum Schubert polynomials or quantum double Schubert polynomials. One major difficulty is the presence of unpredictable signs in the monomial expansion of $\g_w^q(x,y)$. In this paper, we define combinatorial objects called quantum bumpless pipe dreams (QBPDs). They are a generalization of  bumpless pipe dreams, and their weight-generating function gives the quantum double Schubert polynomials, i.e., \begin{equation}
\g_w^q(x,y) = \sum_{P\in \QBPD(w)}\bwt(P)\,,\end{equation}
where $\QBPD(w)$ is the set of QBPDs of $w$ and $\bwt(P)$ is a product of $(x_i-y_j)$'s and $q_i$'s. This is stated precisely in Theorem \ref{thm:qdform}. 
Unfortunately, this formula has internal cancellation, but the combinatorics seems quite natural.

We prove this formula by showing the quantity $\sum\limits_{P\in \QBPD(w)}\bwt(P)$
satisfies a defining transition equation of the quantum double Schubert polynomials (Proposition \ref{prop:stabform}). We do this by constructing four bijections between terms on each side of the recursion, plus a family of additional bijections between terms that cancel out on one side. By showing how these bijections change the binomial weights, we establish the recursion and thus establish the main identity, Theorem ~\ref{thm:qdform}.

We give the necessary background in Section \ref{background}. In Section \ref{quantumbpd}, we define quantum bumpless pipe dreams, establish fundamental combinatorial properties, and state the main theorem. In Section \ref{droopandlift}, we provide a way to generate all QBPDs for a given permutation using droop moves as in \cite{Lam_2021}, as well as new moves called \textit{lift moves}. In Section \ref{stability}, we prove that the quantity $\sum\limits_{P\in \QBPD(w)}\bwt(P)$ satisfies the stability condition, i.e., it does not change under the natural inclusion map $i\colon S_n\to S_{n+1}$, which is needed for our proof of Theorem \ref{thm:qdform}. In Section \ref{proof33}, we prove Theorem \ref{thm:qdform}.
In Section \ref{cancel}, we provide some examples of (partial) cancellations of the binomial weight of QBPDs and provide tables analyzing the cancellations in small degrees. Finally, we provide some future directions in Section \ref{future}.
\subsection*{Acknowledgements}
This project started in the Fall 2023 Lab of Geometry at the University of Michigan, and the initial problem was suggested by George H. Seelinger. We are thankful to him for his mentorship during the project and for helpful comments, as well as suggestions on the writing of this paper. We would also like to thank Sergey Fomin for helpful thoughts and comments. Finally, we thank the anonymous referees for their detailed and useful suggestions.

\section{Background}\label{background}
We use the notation $[n]:=\{1, 2, \dots, n\}.$ Let $S_n$ be the symmetric group on $[n]$, i.e. the set of bijections from $[n]$ to $[n]$. To write a bijection $\sigma\colon [n]\to [n]$, we will use one-line notation, i.e., writing $\sigma(1)\sigma(2)\dots \sigma(n)$. We write $t_{ab}$ for the transposition that swaps $a$ and $b$, and write $s_i$ for the transposition $t_{i, i+1}$. The length of $w\in S_n$, denoted $\ell(w)$, is defined as the minimum number of adjacent transpositions $s_i$ required to express $w$ as their product. Any way to write $w$ as a product of exactly $\ell(w)$ adjacent transpositions is called a \textit{reduced word} for $w$. Let $w_0 := n\: n-1\dots 1$ be the longest permutation in $S_n$.

We write $\Z[x]$ for $\Z[x_1, \dots, x_n]$,  $\Z[x, y]$ for $\Z[x][y_1, \dots, y_n]$, $\Z[x, q]$ for $\Z[x][q_1, \dots, q_{n-1}]$ and $\Z[x, y, q]$ for $\Z[x, y][q_1, \dots, q_{n-1}]$.
\subsection{Rothe Diagrams}
Throughout this paper, we use matrix coordinate notation: $(i, j)$ means the box on row $i$ column $j$. For a permutation $\sigma : [n]\to [n]$, the \textit{Rothe diagram} of $\sigma$ is defined as follows.  The set of boxes $\{(i,\sigma(i)):i\in [n]\}$ are first marked with a dot in the grid. Then, starting from each dot and ending on edges of the grid, vertical lines are drawn downward and horizontal lines are drawn rightward. The resulting figure is the Rothe diagram for $\sigma$. 
\begin{figure}[h!]
        \centering
        \begin{tikzpicture}[scale=\figscale]
            \node at (0.5,-0.3){1};
\node at (1.5,-0.3){2};
\node at (2.5,-0.3){3};
\node at (3.5,-0.3){4};
\node at (4.3,0.5){4};
\node at (4.3,1.5){3};
\node at (4.3,2.5){2};
\node at (4.3,3.5){1};
            \draw (0, 0) grid (4,4);
            
            \draw[blue] (3.5, 3.5)--(4, 3.5); 
            \draw[blue] (3.5, 3.5)--(3.5, 0); 

            \draw[blue] (2.5, 0.5)--(4, 0.5);
            \draw[blue] (2.5, 0.5)--(2.5, 0); 

            \draw[blue] (1.5, 2.5)--(4, 2.5); 
            \draw[blue] (1.5, 2.5)--(1.5, 0); 

            \draw[blue] (0.5, 1.5)--(4, 1.5); 
            \draw[blue] (0.5, 1.5)--(0.5, 0); 
            
            \node[blue] at (3.5,3.5)[circle,fill,inner sep=1.5pt]{};
            \node[blue] at (2.5,0.5)[circle,fill,inner sep=1.5pt]{};
            \node[blue] at (1.5,2.5)[circle,fill,inner sep=1.5pt]{};
            \node[blue] at (0.5,1.5)[circle,fill,inner sep=1.5pt]{};
        \end{tikzpicture}
        \begin{tikzpicture}[scale=\figscale]
\node at (0.5,-0.3){1};
\node at (1.5,-0.3){2};
\node at (2.5,-0.3){3};
\node at (3.5,-0.3){4};
\node at (4.3,0.5){4};
\node at (4.3,1.5){3};
\node at (4.3,2.5){2};
\node at (4.3,3.5){1};
\draw (0, 0) grid (4,4);
\draw[blue](1,1.5) arc (90:180:0.5);
\draw[blue](0.5,0)--(0.5,1);
\draw[blue](2,2.5) arc (90:180:0.5);
\draw[blue](1.5,1)--(1.5,2);
\draw[blue](1,1.5)--(2,1.5);
\draw[blue](1.5,0)--(1.5,1);
\draw[blue](2,2.5)--(3,2.5);
\draw[blue](2,1.5)--(3,1.5);
\draw[blue](3,0.5) arc (90:180:0.5);
\draw[blue](4,3.5) arc (90:180:0.5);
\draw[blue](3.5,2)--(3.5,3);
\draw[blue](3,2.5)--(4,2.5);
\draw[blue](3.5,1)--(3.5,2);
\draw[blue](3,1.5)--(4,1.5);
\draw[blue](3.5,0)--(3.5,1);
\draw[blue](3,0.5)--(4,0.5);
\end{tikzpicture}
        \caption{Rothe diagram for $4213$}
        \label{fig:rothe}
    \end{figure}
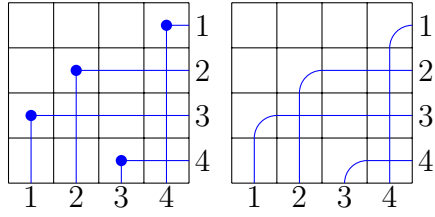
Note that the Rothe diagram for a permutation $\sigma$ can be turned into a bumpless pipe dream by ``smoothing'' the corners into \begin{tikzpicture}[scale=\inlinescale]
\draw (0, 0) grid (1,1);
\draw[blue](1,0.5) arc (90:180:0.5);
\end{tikzpicture} tiles, as in Figure \ref{fig:rothe}. (For the definition of bumpless pipe dream, see \cite{Lam_2021}, or Definition \ref{def:qbpd} below.)

\subsection{Double Schubert polynomials}

Consider the action of $S_n$ on $\Z[x, y]$ by permuting the $y$ variables; in particular, $s_i$ swaps $y_i$ and $y_{i+1}$:\begin{equation}
s_if(x, y_{1}, \dots, y_{i}, y_{i+1}, \dots, y_n) = f(x, y_{1}, \dots, y_{i+1}, y_{i}, \dots, y_n).\end{equation} 
We define divided difference operators $\partial_i^y$ as follows \begin{equation}
    \de_i^y(f) := \frac{f-s_if}{y_i-y_{i+1}}. 
\end{equation}
The divided difference operators $\partial_w^y$ for $w\in S_n$ are defined as follows. Let $ s_{a_1}\cdots s_{a_k}$ be any reduced word for $w$. Then, \begin{equation}\de_{w}^y = \de_{a_1}^y\cdots \de_{a_k}^y.\end{equation} Then, the double Schubert polynomials are defined as follows:
\begin{equation}\g_w(x,y) = \begin{cases}
     \prod\limits_{i+j\leq n}(x_i-y_j), & \text{if }w = w_0,\\
    (-1)^{\ell(w_{0})-\ell(w)}\de_{ww_0}^y\g_{w_0}(x,y), & \text{otherwise}.
\end{cases}\end{equation}
Specializing the $y$ variables to $0$ recovers the Schubert polynomials. 

\subsection{Quantum double Schubert polynomials}
As in \cite{quantum}, we define $E_i^k(x_1, \dots, x_k)$ to be the coefficient of $\lambda^i$ in the characteristic polynomial $\det(1+\lambda G_k)$ where \begin{equation}G_k = \begin{bmatrix}
    x_1 & q_1 & 0 &\dots & 0\\
    -1 & x_2 & q_2 &\dots & 0\\
    0 & -1 & x_3 &\dots & 0\\
    \vdots & \vdots &\vdots&\ddots & \vdots \\
    0 & 0 & 0 &\dots & x_k
\end{bmatrix}.\end{equation}
The quantum double Schubert polynomials are defined as follows. For the longest permutation $w_0$, we have \begin{equation}\g_{w_0}^q(x, y):=\prod_{k = 1}^{n-1} E_k^k(x_1-y_{n-k}, \dots, x_k-y_{n-k}),\end{equation}
and, for any permutation $w$, we have \begin{equation}
\g_{w}^q(x, y) = (-1)^{\ell(w_0)-\ell(w)}\partial_{ww_0}^y\g_{w_0}^q(x,y).    
\end{equation}
This is the definition as in 
\cite{cf}. Specializing the $q$ variables to $0$ recovers the double Schubert polynomials while specializing the $y$ variables to $0$ recovers the quantum Schubert polynomials. Setting both $y$ and $q$ variables to $0$ recovers the Schubert polynomials. 
\begin{theorem}[Monk's rule for quantum double Schubert polynomials \cite{qbsm}]
    For any $k$ and any permutation $w$,
\begin{align*}
    \g_{s_k}^q(x, y)\g_{w}^q(x,y) &= \sum_{\substack{a \leq k < b,\\ \ell(w  t_{ab}) = \ell(w)+1}}\g_{w t_{ab}}^q(x,y)+\sum_{\substack{c \leq k < d,\\ \ell(w t_{cd}) = \ell(w)-\ell(t_{cd})}}q_{cd}\g_{w t_{cd}}^q(x,y)\notag\\
    &+\sum_{i=1}^k(y_{w(i)}-y_i)\g_w^q(x, y),
\end{align*}
where $q_{cd}:= q_{c}q_{c+1}\dots q_{d-1}$.
\label{qdmonk}
\end{theorem}

\section{Quantum bumpless pipe dreams}\label{quantumbpd}
\begin{definition}\label{def:qbpd}
A \textit{quantum bumpless pipe dream} (QBPD) is a tiling of an $n\times n$ grid filled with tiles
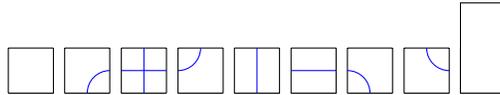
\begin{center}
    \begin{tikzpicture}[scale=\figscale]
\draw (0, 0) grid (1,1);
\end{tikzpicture}
\begin{tikzpicture}[scale=\figscale]
\draw (0, 0) grid (1,1);
\draw[blue](1,0.5) arc (90:180:0.5);
\end{tikzpicture}
\begin{tikzpicture}[scale=\figscale]
\draw (0, 0) grid (1,1);
\draw[blue](0.5,0)--(0.5,1);
\draw[blue](0,0.5)--(1,0.5);
\end{tikzpicture}
\begin{tikzpicture}[scale=\figscale]
\draw (0, 0) grid (1,1);
\draw[blue](0,0.5) arc (-90:0:0.5);
\end{tikzpicture}
\begin{tikzpicture}[scale=\figscale]
\draw (0, 0) grid (1,1);
\draw[blue](0.5,0)--(0.5,1);
\end{tikzpicture}
\begin{tikzpicture}[scale=\figscale]
\draw (0, 0) grid (1,1);
\draw[blue](0,0.5)--(1,0.5);
\end{tikzpicture}
\begin{tikzpicture}[scale=\figscale]
\draw (0, 0) grid (1,1);
\draw[blue](0.5,0) arc (0:90:0.5);
\end{tikzpicture}
\begin{tikzpicture}[scale=\figscale]
\draw (0, 0) grid (1,1);
\draw[blue](0.5,1) arc (180:270:0.5);
\end{tikzpicture}
\begin{tikzpicture}[scale=\figscale]
\draw (0, 0) grid (1,1);
\draw[fill=white,draw=black](0,0) rectangle (1,2);
\end{tikzpicture}
\end{center}

so that
\begin{itemize}
    \item The tiling forms $n$ pipes;
    \item Each pipe starts horizontally at the right edge of the grid and ends vertically at the bottom edge of the grid;    \item The pipes only move upward, downward, or leftward (but not rightward) when moving from the right edge to the bottom edge;
    \item No two pipes cross more than once.
\end{itemize}

 The last tile is a $2\times 1$ domino tile, which occupies two vertically adjacent empty squares in the grid. A (non-quantum) \textit{bumpless pipe dream} (BPD) (as defined in \cite{Lam_2021}), is a QBPD in which the last three tiles above (\begin{tikzpicture}[scale=\inlinescale]
\draw (0, 0) grid (1,1);
\draw[blue](0.5,0) arc (0:90:0.5);
\end{tikzpicture}, 
\begin{tikzpicture}[scale=\inlinescale]
\draw (0, 0) grid (1,1);
\draw[blue](0.5,1) arc (180:270:0.5);
\end{tikzpicture}, and the domino tile) do not appear, so in a BPD pipes only move downward and leftward. The last condition that no two pipes cross more than once is the \textit{reducedness} condition. Throughout this paper, all quantum bumpless pipe dreams are reduced, though in some proofs we may use the term ``reduced quantum bumpless pipe dream'' to stress that the reducedness condition holds when it is not immediately clear why.
\end{definition}
\begin{figure}[h!]
    \centering
    \begin{tikzpicture}[scale=\figscale]
\draw (0, 0) grid (5,5);
\draw[blue](1,4.5) arc (90:180:0.5);
\draw[blue](0.5,3)--(0.5,4);
\draw[blue](0.5,2)--(0.5,3);
\draw[blue](0.5,1)--(0.5,2);
\draw[blue](0.5,0)--(0.5,1);
\draw[blue](1,4.5)--(2,4.5);
\draw[blue](2,3.5) arc (90:180:0.5);
\draw[blue](1.5,3) arc (180:270:0.5);
\draw[blue](2,1.5) arc (90:180:0.5);
\draw[blue](1.5,0)--(1.5,1);
\draw[blue](2.5,4) arc (0:90:0.5);
\draw[blue](2,3.5) arc (-90:0:0.5);
\draw[blue](2,2.5)--(3,2.5);
\draw[blue](2,1.5)--(3,1.5);
\draw[blue](3,0.5) arc (90:180:0.5);
\draw[blue](4,4.5) arc (90:180:0.5);
\draw[blue](3.5,3)--(3.5,4);
\draw[blue](3.5,2)--(3.5,3);
\draw[blue](3,2.5)--(4,2.5);
\draw[blue](3.5,1)--(3.5,2);
\draw[blue](3,1.5)--(4,1.5);
\draw[blue](3.5,0)--(3.5,1);
\draw[blue](3,0.5)--(4,0.5);
\draw[blue](4,4.5)--(5,4.5);
\draw[blue](5,3.5) arc (90:180:0.5);
\draw[blue](4.5,2)--(4.5,3);
\draw[blue](4,2.5)--(5,2.5);
\draw[blue](4.5,1)--(4.5,2);
\draw[blue](4,1.5)--(5,1.5);
\draw[blue](4.5,0)--(4.5,1);
\draw[blue](4,0.5)--(5,0.5);
\end{tikzpicture}
    \caption{A non-example of a QBPD}
    \label{fig:nonexqbpd}
\end{figure}
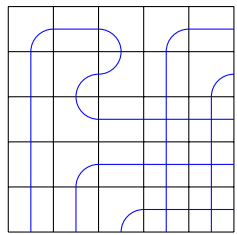
\begin{example}
    Figure \ref{fig:nonexqbpd} shows a non-example of a QBPD. The pipe starting on row $3$ moves rightward in the tiles $(2,2)$ and $(2,3)$, which violates Definition \ref{def:qbpd}.
\end{example}
\begin{definition}\label{def:bwt}
    The \textit{binomial weight} for a QBPD $P$, denoted $\bwt(P)$, is the product of factors contributed by the following rules:
\end{definition}

\begin{itemize}
    \item An empty tile \begin{tikzpicture}[scale=\inlinescale]
\draw (0, 0) grid (1,1);
\end{tikzpicture} on row $i$, column $j$ contributes $x_i-y_j$;
    \item A domino tile
    \begin{tikzpicture}[scale=\inlinescale]
\draw (0, 0) grid (1,1);
\draw[fill=white,draw=black](0,0) rectangle (1,2);
\end{tikzpicture}
    whose upper cell is on row $i$ contributes $q_i$;
    \item A cross tile \begin{tikzpicture}[scale=\inlinescale]
\draw (0, 0) grid (1,1);
\draw[blue](0.5,0)--(0.5,1);
\draw[blue](0,0.5)--(1,0.5);
\end{tikzpicture} on row $i$ where the vertical strand moves upwards contributes $q_i$;
    \item A southwest elbow \begin{tikzpicture}[scale=\inlinescale]
\draw (0, 0) grid (1,1);
\draw[blue](0.5,0) arc (0:90:0.5);
\end{tikzpicture} on row $i$ contributes $-q_i$;
    \item A vertical tile \begin{tikzpicture}[scale=\inlinescale]
\draw (0, 0) grid (1,1);
\draw[blue](0.5,0)--(0.5,1);
\end{tikzpicture} on row $i$ where the strand moves upward contributes $-q_i$.
\end{itemize}
In other words, let $P(i, j)$ denote the cell on row $i$ column $j$, and let \[E(P) := \{(i, j):P(i,j)\text{ is a single empty cell}\},\]
\[Q(P) := \{(i, j):P(i, j)\text{ is the upper cell of a domino or a \begin{tikzpicture}[scale=\inlinescale]
\draw (0, 0) grid (1,1);
\draw[blue](0.5,0)--(0.5,1);
\draw[blue](0,0.5)--(1,0.5);
\end{tikzpicture} tile }\]\[\text{in which the vertical strand moves upward}\},\]
and
\[NQ(P) := \{(i, j):P(i, j)\text{ is a \begin{tikzpicture}[scale=\inlinescale]
\draw (0, 0) grid (1,1);
\draw[blue](0.5,0) arc (0:90:0.5);
\end{tikzpicture} tile or a \begin{tikzpicture}[scale=\inlinescale]
\draw (0, 0) grid (1,1);
\draw[blue](0.5,0)--(0.5,1);
\end{tikzpicture} tile }\]\[\text{in which the strand moves upward}\}.\]
Then, \begin{align}\bwt(P)&:= \prod_{(i, j)\in E(P)}(x_i-y_j)\prod_{(i, j)\in Q(P)}q_i\prod_{(i, j)\in NQ(P)}(-q_i)\notag\\
&= (-1)^{|NQ(P)|}\prod_{(i, j)\in E(P)}(x_i-y_j)\prod_{(i, j)\in Q(P)\cup NQ(P)}q_i \,.\end{align}

A QBPD $P$ is said to be associated with a permutation $w$ if the pipe starting on the right on row $i$ ends up on column $w(i)$. 

Let $\QBPD(w)$ denote the set of QBPDs associated to $w$. Our main result is the following:
\begin{theorem}\label{thm:qdform}
The quantum double Schubert polynomial indexed by $w \in S_n$ is the sum of binomial weights of all QBPDs associated to $w$:
\begin{equation*}
    \g^q_w(x,y) = \sum_{P\in \QBPD(w)}\bwt(P)\,.
\end{equation*}
\end{theorem}
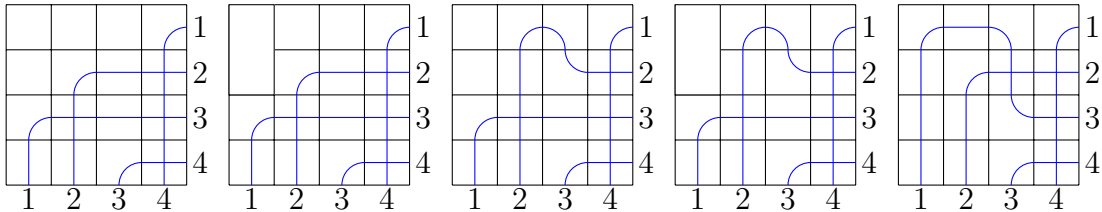
\begin{figure}[!h]
    \centering
\begin{tikzpicture}[scale=\figscale]
\node at (0.5,-0.3){1};
\node at (1.5,-0.3){2};
\node at (2.5,-0.3){3};
\node at (3.5,-0.3){4};
\node at (4.3,0.5){4};
\node at (4.3,1.5){3};
\node at (4.3,2.5){2};
\node at (4.3,3.5){1};
\draw (0, 0) grid (4,4);
\draw[blue](1,1.5) arc (90:180:0.5);
\draw[blue](0.5,0)--(0.5,1);
\draw[blue](2,2.5) arc (90:180:0.5);
\draw[blue](1.5,1)--(1.5,2);
\draw[blue](1,1.5)--(2,1.5);
\draw[blue](1.5,0)--(1.5,1);
\draw[blue](2,2.5)--(3,2.5);
\draw[blue](2,1.5)--(3,1.5);
\draw[blue](3,0.5) arc (90:180:0.5);
\draw[blue](4,3.5) arc (90:180:0.5);
\draw[blue](3.5,2)--(3.5,3);
\draw[blue](3,2.5)--(4,2.5);
\draw[blue](3.5,1)--(3.5,2);
\draw[blue](3,1.5)--(4,1.5);
\draw[blue](3.5,0)--(3.5,1);
\draw[blue](3,0.5)--(4,0.5);
\end{tikzpicture}
\begin{tikzpicture}[scale=\figscale]
\node at (0.5,-0.3){1};
\node at (1.5,-0.3){2};
\node at (2.5,-0.3){3};
\node at (3.5,-0.3){4};
\node at (4.3,0.5){4};
\node at (4.3,1.5){3};
\node at (4.3,2.5){2};
\node at (4.3,3.5){1};
\draw (0, 0) grid (4,4);
\draw[fill=white,draw=black](0,2) rectangle (1,4);
\draw[blue](1,1.5) arc (90:180:0.5);
\draw[blue](0.5,0)--(0.5,1);
\draw[blue](2,2.5) arc (90:180:0.5);
\draw[blue](1.5,1)--(1.5,2);
\draw[blue](1,1.5)--(2,1.5);
\draw[blue](1.5,0)--(1.5,1);
\draw[blue](2,2.5)--(3,2.5);
\draw[blue](2,1.5)--(3,1.5);
\draw[blue](3,0.5) arc (90:180:0.5);
\draw[blue](4,3.5) arc (90:180:0.5);
\draw[blue](3.5,2)--(3.5,3);
\draw[blue](3,2.5)--(4,2.5);
\draw[blue](3.5,1)--(3.5,2);
\draw[blue](3,1.5)--(4,1.5);
\draw[blue](3.5,0)--(3.5,1);
\draw[blue](3,0.5)--(4,0.5);
\end{tikzpicture}
\begin{tikzpicture}[scale=\figscale]
\node at (0.5,-0.3){1};
\node at (1.5,-0.3){2};
\node at (2.5,-0.3){3};
\node at (3.5,-0.3){4};
\node at (4.3,0.5){4};
\node at (4.3,1.5){3};
\node at (4.3,2.5){2};
\node at (4.3,3.5){1};
\draw (0, 0) grid (4,4);
\draw[blue](1,1.5) arc (90:180:0.5);
\draw[blue](0.5,0)--(0.5,1);
\draw[blue](2,3.5) arc (90:180:0.5);
\draw[blue](1.5,2)--(1.5,3);
\draw[blue](1.5,1)--(1.5,2);
\draw[blue](1,1.5)--(2,1.5);
\draw[blue](1.5,0)--(1.5,1);
\draw[blue](2.5,3) arc (0:90:0.5);
\draw[blue](2.5,3) arc (180:270:0.5);
\draw[blue](2,1.5)--(3,1.5);
\draw[blue](3,0.5) arc (90:180:0.5);
\draw[blue](4,3.5) arc (90:180:0.5);
\draw[blue](3.5,2)--(3.5,3);
\draw[blue](3,2.5)--(4,2.5);
\draw[blue](3.5,1)--(3.5,2);
\draw[blue](3,1.5)--(4,1.5);
\draw[blue](3.5,0)--(3.5,1);
\draw[blue](3,0.5)--(4,0.5);
\end{tikzpicture}
\begin{tikzpicture}[scale=\figscale]
\node at (0.5,-0.3){1};
\node at (1.5,-0.3){2};
\node at (2.5,-0.3){3};
\node at (3.5,-0.3){4};
\node at (4.3,0.5){4};
\node at (4.3,1.5){3};
\node at (4.3,2.5){2};
\node at (4.3,3.5){1};
\draw (0, 0) grid (4,4);
\draw[fill=white,draw=black](0,2) rectangle (1,4);
\draw[blue](1,1.5) arc (90:180:0.5);
\draw[blue](0.5,0)--(0.5,1);
\draw[blue](2,3.5) arc (90:180:0.5);
\draw[blue](1.5,2)--(1.5,3);
\draw[blue](1.5,1)--(1.5,2);
\draw[blue](1,1.5)--(2,1.5);
\draw[blue](1.5,0)--(1.5,1);
\draw[blue](2.5,3) arc (0:90:0.5);
\draw[blue](2.5,3) arc (180:270:0.5);
\draw[blue](2,1.5)--(3,1.5);
\draw[blue](3,0.5) arc (90:180:0.5);
\draw[blue](4,3.5) arc (90:180:0.5);
\draw[blue](3.5,2)--(3.5,3);
\draw[blue](3,2.5)--(4,2.5);
\draw[blue](3.5,1)--(3.5,2);
\draw[blue](3,1.5)--(4,1.5);
\draw[blue](3.5,0)--(3.5,1);
\draw[blue](3,0.5)--(4,0.5);
\end{tikzpicture}
\begin{tikzpicture}[scale=\figscale]
\node at (0.5,-0.3){1};
\node at (1.5,-0.3){2};
\node at (2.5,-0.3){3};
\node at (3.5,-0.3){4};
\node at (4.3,0.5){4};
\node at (4.3,1.5){3};
\node at (4.3,2.5){2};
\node at (4.3,3.5){1};
\draw (0, 0) grid (4,4);
\draw[blue](1,3.5) arc (90:180:0.5);
\draw[blue](0.5,2)--(0.5,3);
\draw[blue](0.5,1)--(0.5,2);
\draw[blue](0.5,0)--(0.5,1);
\draw[blue](1,3.5)--(2,3.5);
\draw[blue](2,2.5) arc (90:180:0.5);
\draw[blue](1.5,1)--(1.5,2);
\draw[blue](1.5,0)--(1.5,1);
\draw[blue](2.5,3) arc (0:90:0.5);
\draw[blue](2.5,2)--(2.5,3);
\draw[blue](2,2.5)--(3,2.5);
\draw[blue](2.5,2) arc (180:270:0.5);
\draw[blue](3,0.5) arc (90:180:0.5);
\draw[blue](4,3.5) arc (90:180:0.5);
\draw[blue](3.5,2)--(3.5,3);
\draw[blue](3,2.5)--(4,2.5);
\draw[blue](3.5,1)--(3.5,2);
\draw[blue](3,1.5)--(4,1.5);
\draw[blue](3.5,0)--(3.5,1);
\draw[blue](3,0.5)--(4,0.5);
\end{tikzpicture}
    \caption{QBPDs for $4213$}
    \label{qbpd}
\end{figure}
\begin{example}From Figure \ref{qbpd}, we have\begin{align*}
\g^q_{4213}(x,y) &= (x_1-y_1)(x_1-y_2)(x_1-y_3)(x_2-y_1)+ q_1(x_1-y_2)(x_1-y_3)\\&+(x_1-y_1)(x_2-y_1)(-q_1) + q_1(-q_1)+(-q_1)q_2\,.\end{align*}
\end{example}
As a corollary, we have the following formula.

\begin{corollary}\label{ds}
The quantum Schubert polynomial indexed by $w \in S_n$ is the sum of monomial weights of all QBPDs associated to $w$:
\begin{equation*}\g^q_w(x) = \sum_{P\in \QBPD(w)}\wt(P)\end{equation*}
where \begin{equation*}\wt(P) := \prod_{(i, j)\in E(P)}x_i\prod_{(i, j)\in Q(P)}q_i\prod_{(i, j)\in NQ(P)}(-q_i)\,.\end{equation*}
\end{corollary}
 \subsection{Droop moves and lift moves}\label{droopandlift}

Droop moves on bumpless pipe dreams are defined in  \cite{Lam_2021}. They are moves of the form illustrated in Figure \ref{fig:droop}. Droop moves can be extended to the QBPD setting. We allow a droop move only if the result is a valid and reduced QBPD.
\begin{figure}[h!]
    \centering
    \begin{tikzpicture}[scale=\figscale]
    \draw (0, 0) grid (3,4);
    \draw[blue, ultra thick] (0.5, 1)--(0.5, 3); 
    \draw[blue] (0.5, 0)--(0.5, 1);
    \draw[blue](2,3.5) arc (-90:0:0.5);
    \draw[blue, ultra thick] (1,3.5) arc (90:180:0.5);
    \draw[blue, ultra thick] (1, 3.5)--(2, 3.5);
    \draw[blue] (2, 3.5)--(3, 3.5);
    \draw[blue](0,0.5) arc (-90:0:0.5);
    \draw[->] (3.5,2)--(4.5,2);
    \draw (5, 0) grid (8,4);
    \draw[blue, ultra thick] (7.5, 1)--(7.5, 3); 
    \draw[blue] (7.5, 3)--(7.5, 4);
    \draw[blue](6,0.5) arc (90:180:0.5);
    \draw[blue, ultra thick] (7,0.5) arc (-90:0:0.5);
    \draw[blue, ultra thick] (6, 0.5)--(7, 0.5);
    \draw[blue] (5, 0.5)--(6, 0.5);
    \draw[blue](8,3.5) arc (90:180:0.5);
\end{tikzpicture}
\caption{A droop move (light color indicates possibilities)}
    \label{fig:droop}
\end{figure}
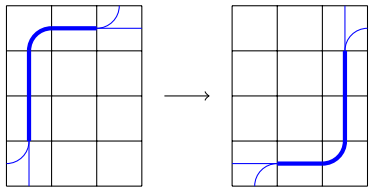

In \cite{Lam_2021}, Lam, Lee, and Shimozono proved that any bumpless pipe dream of a given permutation can be obtained from the Rothe diagram by a sequence of droops. To generate all QBPDs, we introduce other moves called \textit{lift moves}.

The lift moves are moves in which a horizontal segment of a strand is ``lifted up'' into a detour that goes up, moves left, and then moves back down. Figure \ref{fig:lift} shows an example of a lift move. There might be other unpictured pipes in the picture as long as the result is a valid and reduced QBPD.
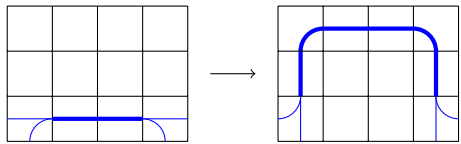
\begin{figure}[!h]
    \centering
    \begin{tikzpicture}[scale=\figscale]
    \draw (0, 0) grid (4,3);
    \draw[blue, ultra thick] (6.5, 1)--(6.5, 2); 
    \draw[blue] (0, 0.5)--(1, 0.5);
    \draw[blue](3.5,0) arc (0:90:0.5);
    \draw[blue, ultra thick] (7,2.5) arc (90:180:0.5);
    \draw[blue, ultra thick] (1, 0.5)--(3, 0.5);
    \draw[blue] (3, 0.5)--(4, 0.5);
    \draw[blue](0.5,0) arc (180:90:0.5);
    \draw[->] (4.5,1.5)--(5.5,1.5);
    \draw (6, 0) grid (10,3);
    \draw[blue, ultra thick] (9.5, 1)--(9.5, 2); 
    \draw[blue] (6.5, 0)--(6.5, 1);
    \draw[blue](6,0.5) arc (-90:0:0.5);
    \draw[blue, ultra thick] (9,2.5) arc (90:0:0.5);
    \draw[blue, ultra thick] (7, 2.5)--(9, 2.5);
    \draw[blue] (9.5, 0)--(9.5, 1);
    \draw[blue](10, 0.5) arc (-90:-180:0.5);
\end{tikzpicture}
\caption{A lift move (light color indicates possibilities)}
    \label{fig:lift}
\end{figure}

A QBPD is said to be \textit{unpaired} if it has no domino tile. To generate all QBPDs, we can generate all unpaired QBPDs and find all ways to pair empty boxes into dominos. Note that droops moves and lift moves preserve the permutation associated to the QBPD. It turns out (Lemma \ref{lem:droopliftgen} below) that all unpaired QBPD for a permutation can be generated from the Rothe diagram using only droop and lift moves. To prove this, we first need some technical definitions and lemmas.

\begin{definition}\label{def:simple}
    A \textit{leftmost region} is a region given by all columns to the left of a given column. A \textit{simple region} is a region in which the tiles containing pipes are all \begin{tikzpicture}[scale=\inlinescale]
\draw (0, 0) grid (1,1);
\draw[blue](0,0.5)--(1,0.5);
\end{tikzpicture}, \begin{tikzpicture}[scale=\inlinescale]
\draw (0, 0) grid (1,1);
\draw[blue](0.5,0)--(0.5,1);
\end{tikzpicture}, \begin{tikzpicture}[scale=\inlinescale]
\draw (0, 0) grid (1,1);
\draw[blue](0.5,0)--(0.5,1);
\draw[blue](0,0.5)--(1,0.5);
\end{tikzpicture}, or \begin{tikzpicture}[scale=\inlinescale]
\draw (0, 0) grid (1,1);
\draw[blue](1,0.5) arc (90:180:0.5);
\end{tikzpicture} tiles. A strand is said to have the \textit{simple form} if it goes straight left, takes a \begin{tikzpicture}[scale=\inlinescale]
\draw (0, 0) grid (1,1);
\draw[blue](1,0.5) arc (90:180:0.5);
\end{tikzpicture} turn, then goes straight down.
\end{definition}
\begin{lemma}\label{simple}
    If \begin{tikzpicture}[scale=\inlinescale]
\draw (0, 0) grid (1,1);
\draw[blue](0,0.5) arc (-90:0:0.5);
\end{tikzpicture} and \begin{tikzpicture}[scale=\inlinescale]
\draw (0, 0) grid (1,1);
\draw[blue](0.5,0) arc (0:90:0.5);
\end{tikzpicture} tiles do not exist in a leftmost region of a QBPD, then that region of the QBPD must be simple. Furthermore, every strand has the simple form in that leftmost region.

 \begin{proof}
If \begin{tikzpicture}[scale=\inlinescale]
\draw (0, 0) grid (1,1);
\draw[blue](0,0.5) arc (-90:0:0.5);
\end{tikzpicture} and \begin{tikzpicture}[scale=\inlinescale]
\draw (0, 0) grid (1,1);
\draw[blue](0.5,0) arc (0:90:0.5);
\end{tikzpicture} do not exist in some leftmost region of the QBPD, then there is also no \begin{tikzpicture}[scale=\inlinescale]
\draw (0, 0) grid (1,1);
\draw[blue](0.5,1) arc (180:270:0.5);
\end{tikzpicture} tile, since if a \begin{tikzpicture}[scale=\inlinescale]
\draw (0, 0) grid (1,1);
\draw[blue](0.5,1) arc (180:270:0.5);
\end{tikzpicture} exists, then the pipe of that tile must move upward and eventually take a left to produce a \begin{tikzpicture}[scale=\inlinescale]
\draw (0, 0) grid (1,1);
\draw[blue](0.5,0) arc (0:90:0.5);
\end{tikzpicture} tile within the region. Thus, the only possible elbow tile is the \begin{tikzpicture}[scale=\inlinescale]
\draw (0, 0) grid (1,1);
\draw[blue](1,0.5) arc (90:180:0.5);
\end{tikzpicture} tile, and all strands must have the ``simple'' form in which it only has a single turn. In particular, if the region is the entire QBPD, then the entire QBPD has the ``simple'' form and it is exactly the Rothe diagram.
 \end{proof}
\end{lemma}
\begin{lemma}\label{lem:droopliftgen}
    All unpaired QBPDs can be generated from the Rothe diagram using a sequence of droop and lift moves.
    \begin{proof}
        It suffices to show that, given any QBPD, we can use a sequence of valid inverse droop and inverse lift moves to get back to the Rothe diagram. Given a QBPD that is not the Rothe diagram, we find the leftmost \begin{tikzpicture}[scale=\inlinescale]
\draw (0, 0) grid (1,1);
\draw[blue](0,0.5) arc (-90:0:0.5);
\end{tikzpicture} or \begin{tikzpicture}[scale=\inlinescale]
\draw (0, 0) grid (1,1);
\draw[blue](0.5,0) arc (0:90:0.5);
\end{tikzpicture}
 tile. If no such tile exists, then the QBPD is just the Rothe diagram by Lemma \ref{simple}. Otherwise, if there are ties, take the topmost one.
 
Now look at the region strictly to the left of that leftmost chosen tile, this region is free of \begin{tikzpicture}[scale=\inlinescale]
\draw (0, 0) grid (1,1);
\draw[blue](0,0.5) arc (-90:0:0.5);
\end{tikzpicture} and \begin{tikzpicture}[scale=\inlinescale]
\draw (0, 0) grid (1,1);
\draw[blue](0.5,0) arc (0:90:0.5);
\end{tikzpicture}
 tiles by definition, and thus is simple by Lemma \ref{simple}. Let $R$ denotes this simple region. By Lemma \ref{simple}, every strand in $R$ has the simple form. 
 
     We first consider the case where the leftmost chosen tile is \begin{tikzpicture}[scale=\inlinescale]
\draw (0, 0) grid (1,1);
\draw[blue](0,0.5) arc (-90:0:0.5);
\end{tikzpicture}. Call this tile $e$. Then, tracing this pipe to the left, it eventually takes a \begin{tikzpicture}[scale=\inlinescale]
\draw (0, 0) grid (1,1);
\draw[blue](1,0.5) arc (90:180:0.5);
\end{tikzpicture} turn---call this elbow $a$---and go straight down since this is $R$. Tracing this pipe up from tile $e$, it eventually takes a \begin{tikzpicture}[scale=\inlinescale]
\draw (0, 0) grid (1,1);
\draw[blue](1,0.5) arc (90:180:0.5);
\end{tikzpicture} turn---call this elbow $b$. Consider the rectangle bounded by $a,b, e$. We claim we can undroop $e$ into the topleft corner of this rectangle. Call this topleft corner $c$. Note that every other pipe in $R$ has the simple form, so its behavior in $R$ is determined by the location of its \begin{tikzpicture}[scale=\inlinescale]
\draw (0, 0) grid (1,1);
\draw[blue](1,0.5) arc (90:180:0.5);
\end{tikzpicture} tile, which cannot be on the row of $b$ or column of $a$ (since the pipe would run into either $a$ or $b$ otherwise). This means there are no pipes moving horizontally in the $b\to c$ horizontal segment, there are no pipes moving vertically in the $c\to a$ horizontal segment, and $c$ is empty. Thus, the undroop does not conflict with other pipes. After undrooping, the changed strand is simple in $R$ and all other strands are still simple, so in $R$, they cross at most once. There could be at most one other crossing since crossings out of $R$ exist before the undroop. Then if any two pipes cross twice after the undroop, they must cross once before the undroop, which is a contradiction since the parity of the number of crossings cannot change.

Now consider the case where the leftmost chosen tile is \begin{tikzpicture}[scale=\inlinescale]
\draw (0, 0) grid (1,1);
\draw[blue](0.5,0) arc (0:90:0.5);
\end{tikzpicture}. Call this tile $e$.
In this case, trace the pipe down and find the \begin{tikzpicture}[scale=\inlinescale]
\draw (0, 0) grid (1,1);
\draw[blue](0.5,1) arc (180:270:0.5);
\end{tikzpicture} tile and call this tile $a$. Trace the pipe to the left and find the \begin{tikzpicture}[scale=\inlinescale]
\draw (0, 0) grid (1,1);
\draw[blue](1,0.5) arc (90:180:0.5);
\end{tikzpicture} turn. Call this tile $b$. After $b$, it goes straight down since this is in $R$. Make the inverse lift by unlifting the $e\to b$ horizontal segment down to the row of $a$, i.e., making the pipe at $a$ move to the left until it hits the column of $b$ instead of moving up to $e$. As in the previous case, other pipes in $R$ have the simple form and the \begin{tikzpicture}[scale=\inlinescale]
\draw (0, 0) grid (1,1);
\draw[blue](1,0.5) arc (90:180:0.5);
\end{tikzpicture} tile of another strand in $R$ cannot be on the row of $a$. So, the inverse lift does not conflict with other pipes. After inverse lifting, the strand is simple in $R$, so a similar argument to the previous case shows it is reduced.
 
 After each inverse droop or inverse lift above, the number of \begin{tikzpicture}[scale=\inlinescale]
\draw (0, 0) grid (1,1);
\draw[blue](0.5,0) arc (0:90:0.5);
\end{tikzpicture} or \begin{tikzpicture}[scale=\inlinescale]
\draw (0, 0) grid (1,1);
\draw[blue](0,0.5) arc (-90:0:0.5);
\end{tikzpicture} tiles decreases, so we eventually arrive at the Rothe diagram.
    \end{proof}
\end{lemma}
\subsection{Stability}\label{stability}
Given a QBPD of $w\in S_n$, we can think of $w$ as being in $S_{n+1}$. In terms of QBPDs, we can extend an $n\times n$ QBPD $P$ to a $(n+1)\times (n+1)$ QBPD as in Figure \ref{fig:sta}. 
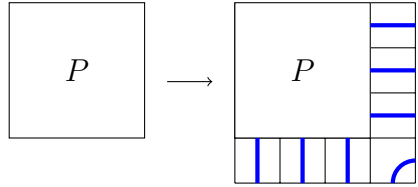
\begin{figure}[h!]
    \centering
    \begin{tikzpicture}[scale=\figscale]
        \draw (0,0) rectangle (3,3);
        \node at (1.5,1.5) {$P$};
        \draw[->] (3.5,1.25)--(4.5,1.25);
        \draw (5,-1) grid (9, 3);
        \draw[fill=white] (5,0) rectangle (8,3);
        \node at (6.5,1.5) {$P$};
        \draw[blue, ultra thick] (7.5,0)--(7.5,-1); 
        \draw[blue, ultra thick] (5.5,0)--(5.5,-1); 
        \draw[blue, ultra thick] (6.5,0)--(6.5,-1); 
        \draw[blue, ultra thick] (8,0.5)--(9,0.5); 
        \draw[blue, ultra thick] (8,1.5)--(9,1.5);  
        \draw[blue, ultra thick] (8,2.5)--(9,2.5); 
        \draw[blue, ultra thick] (8.5,-1) arc (180:90:0.5);
    \end{tikzpicture}
    \caption{Extending a $n\times n$ QBPD to a $(n+1)\times (n+1)$ QBPD}
    \label{fig:sta}
\end{figure}

The reverse is also possible:
\begin{lemma}\label{stab}
    Given $w\in S_{n+1}$ such that $w(n+1)= n+1$, for any $(n+1)\times (n+1)$ QBPD of $w$, we can restrict it to an $n\times n$ QBPD of $w$ restricted to $S_n$.
\end{lemma}
\begin{remark}\label{uplast}
    No pipes move upward in the rightmost column of any valid QBPD. This is because if a pipe moves upward in the rightmost column, eventually, it turns left and produces a \begin{tikzpicture}[scale=\inlinescale]
\draw (0, 0) grid (1,1);
\draw[blue](0.5,0) arc (0:90:0.5);
\end{tikzpicture} tile on the rightmost column, which blocks the pipe that starts on that row.
\end{remark}
\begin{proof}[Proof of Lemma \ref{stab}]
    The $(n+1, n+1)$ tile is a \begin{tikzpicture}[scale=\inlinescale]
        \draw (0, 0) grid (1,1);
        \draw[blue](1,0.5) arc (90:180:0.5);
    \end{tikzpicture} since $\sigma(n+1) = n+1$. If there is a tile in the rightmost column except for the bottom right corner that is not horizontal, let $(k, n+1)$ be the lowest such tile (largest $k$). Then, this pipe cannot move up by Remark \ref{uplast} so it must move down, but then it will run into the $(n+1, n+1)$ elbow, which is a contradiction. Thus, all tiles in the rightmost column except the bottom right corner are horizontal. Similar reasoning shows all tiles in the last row except the bottom right corner are vertical. Thus, we can reduce the QBPD to $n\times n$.
\end{proof}
\subsection{Proof of Theorem \ref{thm:qdform}}\label{proof33}
We write $\sigma t_{ab}\gtrdot\sigma$ for $\ell(\sigma t_{ab}) = \ell(\sigma)+1$ and $\sigma  t_{cd}\lhd\sigma$ for $\ell(\sigma t_{cd}) = \ell(\sigma)-\ell(t_{cd})$ to simplify notation. 
\begin{remark}\label{cond}
    Let $\sigma$ be a permutation and $a < b$. Then,
    \begin{itemize}
        \item $\sigma t_{ab}\gtrdot \sigma$ if and only if $ \sigma(a) < \sigma(b)$  and for all $a < k < b$, $\sigma(k) < \sigma(a)$ or $\sigma(k) > \sigma(b)$.
        \item $\sigma t_{ab}\lhd \sigma$ if and only if $\sigma(a) > \sigma(b)$ and for all $a < k < b$, $\sigma(a) > \sigma(k) > \sigma(b)$.
    \end{itemize}
\end{remark}
\begin{definition}
    Let $S_\infty$ denote the set of permutations of $\Z_{>0}$ that fix all but finitely many elements. 
\end{definition}
\begin{proposition}\label{prop:stabform}
    Let $\pi$ be a permutation in $S_{\infty}$ that is not the identity, and let $n$ be the largest number such that $\pi(n) \neq n$. Let $a = \pi^{-1}(n)$ and $b > a$ be such that \begin{equation}\pi(b) = \max\limits_{n\geq i > a}\pi(i),\end{equation}i.e., the number that got mapped to the largest among those after $a$, except those after $n$.
Let $\sigma = \pi t_{ab}$. Then, we have a transition equation for quantum double Schubert polynomials:
\begin{equation}
\begin{split}\label{eq:transition}
\g_\pi^q(x,y) &= (x_a-y_{\sigma(a)})\g^q_{\sigma}(x,y) +\sum_{\substack{ c < a,\\ \sigma t_{ca}\gtrdot\sigma}}\g^q_{\sigma t_{ca}}(x,y)\\&-\sum_{\substack{a < c,\\ \sigma t_{ac}\lhd\sigma}}q_{ac}\g^q_{\sigma t_{ac}}(x,y)+\sum_{\substack{c < a,\\ \sigma t_{ca}\lhd\sigma}}q_{ca}\g^q_{\sigma t_{ca}}(x,y).\end{split}
\end{equation}
Furthermore, \eqref{eq:transition}, along with the base case $\g_{\id}^q(x,y)= 1$, is enough to uniquely determine all the quantum double Schubert polynomials.
\begin{proof}
    Note that $b$ is the only number greater than $a$ (and no more than $n$) such that $\sigma(b) > \sigma(a)$. Thus $\pi = \sigma t_{ab}\gtrdot \sigma$, and so $b$ is the only number greater than $a$ such that $\sigma t_{ab}\gtrdot \sigma$. By Theorem \ref{qdmonk},
\begin{equation}\g_{s_a}^q(x,y)\g_{\sigma}^q(x,y) = \sum_{\substack{u \leq a < v,\\ \sigma t_{uv}\gtrdot\sigma}}\g_{w t_{uv}}^q(x,y)+\sum_{\substack{c \leq a < d,\\ \sigma t_{cd}\lhd\sigma}}q_{cd}\g_{w t_{cd}}^q(x,y)
+\sum_{i=1}^a(y_{w(i)}-y_i)\g_w^q(x, y)\,,\end{equation}
and \begin{equation}\g_{s_{a-1}}^q(x,y)\g_{\sigma}^q(x,y) = \sum_{\substack{u < a \leq v,\\ \sigma t_{uv}\gtrdot\sigma}}\g_{w t_{uv}}^q(x,y)+\sum_{\substack{c < a \leq d,\\ \sigma t_{cd}\lhd\sigma}}q_{cd}\g_{w t_{cd}}^q(x,y)+\sum_{i=1}^{a-1}(y_{w(i)}-y_i)\g_w^q(x, y).\end{equation}
Computing $\g_{s_a}^q(x,y)\g_w^q(x, y) - \g_{s_{a-1}}^q(x,y)\g_w^q(x, y)$ and rearranging, we get \begin{align}
(x_a-y_{\sigma(a)})\g^q_{\sigma}(x,y) &= \sum_{\substack{ a < v,\\ \sigma t_{av}\gtrdot\sigma}}\g^q_{\sigma t_{av}}(x,y)-\sum_{\substack{ u < a\\ \sigma t_{ua}\gtrdot\sigma}}\g^q_{\sigma t_{ua}}(x,y)\notag\\
&+\sum_{\substack{a < d,\\ \sigma t_{ad}\lhd\sigma}}q_{ad}\g^q_{\sigma t_{ad}}(x,y)-\sum_{\substack{c < a,\\ \sigma t_{ca}\lhd\sigma}}q_{ca}\g^q_{\sigma t_{ca}}(x,y)\,.
\end{align}
By the observation that $b$ is the only number greater than $a$ such that $\sigma t_{ab} \gtrdot \sigma$, 
we have  
\begin{align}
(x_a-y_{\sigma(a)})\g^q_{\sigma}(x,y) &= \g_{\pi}^q(x,y)-\sum_{\substack{ c < a\\ \sigma t_{ca}\gtrdot\sigma}}\g^q_{\sigma t_{ca}}(x,y)\notag\\
&+\sum_{\substack{a < c,\\ \sigma t_{ac}\lhd\sigma}}q_{ac}\g^q_{\sigma t_{ac}}(x,y)-\sum_{\substack{c < a,\\ \sigma t_{ca}\lhd\sigma}}q_{ca}\g^q_{\sigma t_{ca}}(x,y)\,,
\end{align}
and rearranging gives \eqref{eq:transition}.

For a permutation $\pi$ that is not the identity permutation, let $m(\pi)$ denote  the largest $n$ such that $\pi(n)\neq n$  and let $p(\pi)$ denote $\pi^{-1}(m(\pi))$. Then for any permutation $\tau$ that appears on the RHS, we have $m(\tau) < m(\pi)$ or $m(\tau) = m(\pi)$ and $p(\tau) > p(\pi)$ unless $\tau = \id$. Thus, \eqref{eq:transition}, along with the base case $\g_{\id}^q(x,y)= 1$, is enough to uniquely determine all quantum double Schubert polynomials.
\end{proof}
\end{proposition}
\begin{remark}
    The usual way to obtain a transition equation from Monk's rule is by taking $a$ to be the last descent of $\pi$. Here, we take $a = \pi^{-1}(n)$ instead. This choice is deliberate: our proof of Theorem \ref{thm:qdform} works with this version of transition equation and not the usual version.
\end{remark}

We give a bijective proof to show that polynomials generated by QBPDs with binomial weights satisfying \eqref{eq:transition}. There will be $4$ bijections, each matching some subset of terms on the LHS of \eqref{eq:transition} with one of the four terms on the RHS of \eqref{eq:transition}, except for portions of the last term \begin{equation}\sum\limits_{\substack{c < a,\\ \sigma t_{ca}\lhd\sigma}}q_{ca}\g^q_{\sigma t_{ca}}(x,y)\end{equation} that will remain, which will be shown to cancel out using another series of bijections. The first two bijections (to be called $\phi_A$ and $\phi_B$) are essentially the bijection in \cite{WEIGANDT2021105470}, which could also be viewed as a special case of the algorithm given in \cite{huang}. 

Fix $\pi\in S_\infty$ and corresponding $n, \sigma, a, b$ as above. Note that we can think of $\pi$ as a permutation in $S_n$. By Lemma \ref{stab}, the polynomial generated by the QBPDs of $\pi$ does not depend on if we think of $\pi$ as being in $S_n$ or in $S_N$ for any $N\ge n$. Furthermore, for any other permutation  $\tau$ appearing in \eqref{eq:transition}, $m(\tau)\le n$, and thus by Lemma \ref{stab} we can think of $\tau$ as a permutation $S_n$ without changing  the polynomial generated by the QBPDs of $\tau$. Let $m = \pi(b) = \sigma(a)$. For reasons that will become clear later, it is helpful to define $S := \{c < a : \sigma t_{ca} \lhd \sigma\}$. Order the elements of $S$ so that $S = \{a_1 > a_2 > \dots > a_k\}$. For convenience let $a_0:= a$. Let $p_i := \sigma(a_i)$, in particular $p_0 = \sigma(a) = m$. In summary, we have the following setup:
\begin{setup}\label{su}
 $\pi\in S_\infty$, $n$ is the largest number such that $\pi(n)\ne n$, and thus $\pi\in S_n$. Set $a = \pi^{-1}(n) < n$ so $\pi(a) = n$, and \begin{equation}\pi(b) = \max_{a < i \leq n}\pi(i).\end{equation}
 Set $\sigma = \pi t_{ab}\in S_n$, and $m = \pi(b) = \sigma(a)$. Note that $b$ is the only number greater than $a$ such that $\sigma(b) > m$. As above, we have \begin{equation}S = \{c < a : \sigma t_{ca} \lhd \sigma\} = \{a_1 > a_2 > \dots > a_k\},\end{equation}
 and $a_0 = a$, $p_i = \sigma(a_i)$ for $0\leq i \leq k$, in particular, $p_0 = \sigma(a) = m$.
\end{setup}
Recall the matrix coordinate notation: $(i, j)$ means the tile on row $i$ column $j$ and $[a, b]\times [c, d]$ represents the rectangle consisting of row $a$ to $b$ and column $c$ to $d$.
\begin{proposition}\label{prop:configpi}
    In Setup \ref{su}, any QBPD for $\pi$ has one of the $4$ configurations in Figure \ref{fig:configpi} when restricted to $[a, n]\times[m, n]$. 
    \begin{figure}[h!]
        \centering
        \begin{tikzpicture}[scale=\figscale]
            \node at (0.5, -0.3) {$m$};
            \node at (6.5, -0.3) {$n$};
            \node at (7.3, 5.5) {$a$};
            \node at (7.3, 1.5) {$b$};
            \node at (3.5, -0.5) {$A$};
            \draw (0, 0) grid (7,6);
            \foreach \i in {1,2,3}
                {\draw[green] (\i.5, 0)--(\i.5, 6);
                \draw[green] (0, 1+\i.5)--(7,1+\i.5);
                } 
            \draw[green] (0, 0.5)--(7, 0.5);   
            \draw[green] (5.5, 0)--(5.5, 6);
            \draw[green] (4.5, 0)--(4.5, 6);
            \draw[blue, ultra thick] (1, 1.5)--(7, 1.5);
            \draw [blue, ultra thick](1,1.5) arc (90:180:0.5);
            \draw[blue, ultra thick] (0.5, 0)--(0.5, 1);
            \draw[blue, ultra thick] (6.5, 0)--(6.5, 5);
            \draw[blue, ultra thick] (6.5, 5) arc (180:90:0.5);
        \end{tikzpicture}
        \begin{tikzpicture}[scale=\figscale]
            \node at (0.5, -0.3) {$m$};
            \node at (6.5, -0.3) {$n$};
            \node at (7.3, 5.5) {$a$};
            \node at (7.3, 1.5) {$b$};
            \node at (3.5, -0.5) {$B$};
            \draw[blue, ultra thick] (0.5, 6) arc (0:-90:0.5);
            \draw (0, 0) grid (7,6);
            \foreach \i in {1,2,3}
                {\draw[green] (\i.5, 0)--(\i.5, 6);
                \draw[green] (0, 1+\i.5)--(7,1+\i.5);
                } 
            \draw[green] (0, 0.5)--(7, 0.5);
            \draw[green] (5.5, 0)--(5.5, 6);
            \draw[green] (4.5, 0)--(4.5, 6);
            \draw[blue, ultra thick] (1, 1.5)--(7, 1.5);
            \draw [blue, ultra thick](1,1.5) arc (90:180:0.5);
            \draw[blue, ultra thick] (0.5, 0)--(0.5, 1);
            \draw[blue, ultra thick] (6.5, 0)--(6.5, 5);
            \draw[blue, ultra thick] (6.5, 5) arc (180:90:0.5);
        \end{tikzpicture}

        \begin{tikzpicture}[scale=\figscale]
            \node at (0.5, -0.3) {$m$};
            \node at (6.5, -0.3) {$n$};
            \node at (7.3, 5.5) {$a$};
            \node at (7.3, 1.5) {$b$};
            \node at (7.3, 3.5) {$c$};
            \node at (3.5, -0.5) {$C$};
            
            \draw (0, 0) grid (7,6);
            \draw[white] (1,7) --(2,7);
            \foreach \i in {1,2,3}
                {\draw[green] (\i.5, 0)--(\i.5, 6);
                \draw[green] (0, 1+\i.5)--(7,1+\i.5);
                } 
            \draw[green] (0, 0.5)--(7, 0.5);  
            \draw[green] (5.5, 0)--(5.5, 6);
            \draw[green] (4.5, 0)--(4.5, 6);
            \draw [blue, ultra thick] (0.5,4)--(0.5,5);
            \draw [blue, ultra thick] (1,3.5)--(7,3.5);
            \draw[blue, ultra thick] (1, 1.5)--(7, 1.5);
            \draw[fill=white] (0,3) rectangle (1,4);
            \draw[blue] (0.5, 5) arc (0:90:0.5);
            \draw[blue] (0.5, 5)--(0.5,6);
            \draw[blue, ultra thick] (0.5, 4) arc (180:270:0.5);
            \draw [blue, ultra thick](1,1.5) arc (90:180:0.5);
            \draw[blue, ultra thick] (0.5, 0)--(0.5, 1);
            \draw[blue, ultra thick] (6.5, 0)--(6.5, 5);
            \draw[blue, ultra thick] (6.5, 5) arc (180:90:0.5);
        \end{tikzpicture}
        \begin{tikzpicture}[scale=\figscale]
            \node at (0.5, -0.3) {$m$};
            \node at (6.5, -0.3) {$n$};
            \node at (7.3, 5.5) {$a$};
            \node at (7.3, 1.5) {$b$};
            \node at (3.5, -0.5) {$D$};
            \draw (0, 0) grid (7,6);
            \draw[fill=white] (0,5) rectangle (1,7);
            \foreach \i in {1,2,3}
                {\draw[green] (\i.5, 0)--(\i.5, 6);
                \draw[green] (0, 1+\i.5)--(7,1+\i.5);
                } 
            \draw[green] (0, 0.5)--(7, 0.5);  
            \draw[green] (5.5, 0)--(5.5, 6);
            \draw[green] (4.5, 0)--(4.5, 6);
            \draw[blue, ultra thick] (1, 1.5)--(7, 1.5);
            \draw [blue, ultra thick](1,1.5) arc (90:180:0.5);
            \draw[blue, ultra thick] (0.5, 0)--(0.5, 1);
            \draw[blue, ultra thick] (6.5, 0)--(6.5, 5);
            \draw[blue, ultra thick] (6.5, 5) arc (180:90:0.5);
        \end{tikzpicture}
        \caption{Configurations for QBPDs of $\pi$. The number of green pipes may differ, but they must go straight as in the figure. Thin blue (in C) indicates possibilities. The numbers written on the left and at the bottom of each diagram are row numbers and column numbers.
        }
        \label{fig:configpi}
    \end{figure}
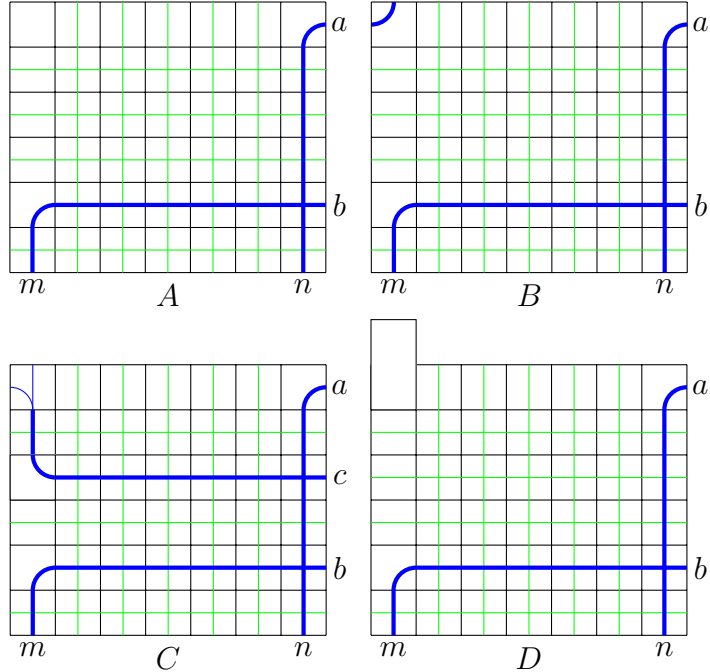
\begin{proof}
    The Rothe diagram has configuration $A$, where the green pipes are straight since $\pi(a)= n$ and $m = \pi(b) > \pi(d)$ for all $d > a, d\neq b$. Since all unpaired QBPDs can be generated from the Rothe diagram by a sequence of lift moves and droop moves, we only need to analyze what lift moves and droop moves can do to the $[a,n]\times [m, n]$ region. Droop moves need an empty tile to droop onto, so if it affects this region, then it must droop onto the empty tile at $(a,m)$, which gives configuration $B$. Lift moves need to turn a \begin{tikzpicture}[scale=\inlinescale]
\draw (0, 0) grid (1,1);
\draw[blue](0,0.5)--(1,0.5);
\end{tikzpicture} tile or a \begin{tikzpicture}[scale=\inlinescale]
\draw (0, 0) grid (1,1);
\draw[blue](0,0.5) arc (90:0:0.5);
\end{tikzpicture} tile upward, so from configuration $A$ only the \begin{tikzpicture}[scale=\inlinescale]
\draw (0, 0) grid (1,1);
\draw[blue](0,0.5)--(1,0.5);
\end{tikzpicture} tile on column $m$ and row strictly between $a$ and $b$ can be turned up. Thus, a droop move that changes the configuration will turn a tile at $(c, m)$ up, for some $a < c < b$, which gives configuration $C$. Once we reach configuration $B$, droop moves cannot change this region since there are no more empty tiles. Lift moves cannot change the region either, since the elbow at $(a, m)$ blocks all lift moves that turn any \begin{tikzpicture}[scale=\inlinescale]
\draw (0, 0) grid (1,1);
\draw[blue](0,0.5)--(1,0.5);
\end{tikzpicture} tiles directly below it upward. Similarly, once we reach configuration $C$, no more droop moves or lift moves can change the configuration (except for possibly a lift move that turns the \begin{tikzpicture}[scale=\inlinescale]
    \draw (0, 0) grid (1,1);
    \draw[blue](0,0.5) arc (90:0:0.5);
    \end{tikzpicture} tile at $(a, m)$ into a \begin{tikzpicture}[scale=\inlinescale]\draw (0, 0) grid (1,1);
    \draw[blue](0.5,0)--(0.5,1);
    \end{tikzpicture} tile.) Finally, since droop and lift moves only generate unpaired QBPDs, a fourth possible configuration is when we have configuration $A$, but the empty tile on $(a,m)$ gets paired into a domino, which is configuration $D$.
\end{proof}
\end{proposition}
\begin{proposition}\label{prop:configsig}
    In Setup \ref{su}, any QBPD for $\sigma$ has the configuration of Figure \ref{fig:cfgsig} when restricted to $[a,n]\times [m, n]$:
    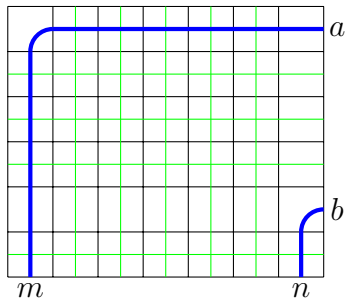
\begin{figure}[h!]
        \centering
        \begin{tikzpicture}[scale=\figscale]
            \node at (0.5, -0.3) {$m$};
            \node at (6.5, -0.3) {$n$};
            \node at (7.3, 5.5) {$a$};
            \node at (7.3, 1.5) {$b$};
            \draw (0, 0) grid (7,6);
            \foreach \i in {1,2,3}
                {\draw[green] (\i.5, 0)--(\i.5, 6);
                \draw[green] (0, 1+\i.5)--(7,1+\i.5);
                } 
            \draw[green] (0, 0.5)--(7, 0.5);   
            \draw[green] (5.5, 0)--(5.5, 6);
            \draw[green] (4.5, 0)--(4.5, 6);
            \draw[blue, ultra thick] (1, 5.5)--(7, 5.5);
            \draw [blue, ultra thick](1,5.5) arc (90:180:0.5);
            \draw[blue, ultra thick] (0.5, 0)--(0.5, 5);
            \draw[blue, ultra thick] (6.5, 0)--(6.5, 1);
            \draw[blue, ultra thick] (6.5, 1) arc (180:90:0.5);
        \end{tikzpicture}
        \caption{Configuration of QBPDs of $\sigma$}
        \label{fig:cfgsig}
    \end{figure}
    \begin{proof}
        This is the configuration for the Rothe diagram of $\sigma$, by a similar argument as in the proof of Proposition \ref{prop:configpi}. There are no empty tiles, so droop moves cannot affect this region. Lift moves need to turn a \begin{tikzpicture}[scale=\inlinescale]
\draw (0, 0) grid (1,1);
\draw[blue](0,0.5)--(1,0.5);
\end{tikzpicture} tile or a \begin{tikzpicture}[scale=\inlinescale]
\draw (0, 0) grid (1,1);
\draw[blue](0,0.5) arc (90:0:0.5);
\end{tikzpicture} tile upward, but all such tiles in this region lie in the rightmost column, so they cannot move upward by Remark \ref{uplast}. Thus, lift moves cannot change this region.
    \end{proof}
\end{proposition}
\begin{proposition}\label{prop:sigcagtr}
    In Setup \ref{su}, for $c < a$ such that $\sigma t_{ca} \gtrdot \sigma$, all QBPDs of $\sigma t_{ca}$ have the configuration in Figure \ref{fig:cfgsigca} when restricted to $[a, n]\times [m,n]$.
    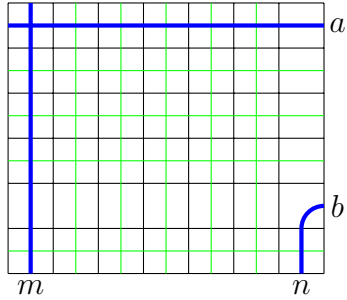
\begin{figure}[h!]
        \centering
        \begin{tikzpicture}[scale=\figscale]
            \node at (0.5, -0.3) {$m$};
            \node at (6.5, -0.3) {$n$};
            \node at (7.3, 5.5) {$a$};
            \node at (7.3, 1.5) {$b$};
            \draw (0, 0) grid (7,6);
            \foreach \i in {1,2,3}
                {\draw[green] (\i.5, 0)--(\i.5, 6);
                \draw[green] (0, 1+\i.5)--(7,1+\i.5);
                } 
            \draw[green] (0, 0.5)--(7, 0.5);   
            \draw[green] (5.5, 0)--(5.5, 6);
            \draw[green] (4.5, 0)--(4.5, 6);
            \draw[blue, ultra thick] (0, 5.5)--(7, 5.5);
            \draw[blue, ultra thick] (0.5, 0)--(0.5, 6);
            \draw[blue, ultra thick] (6.5, 0)--(6.5, 1);
            \draw[blue, ultra thick] (6.5, 1) arc (180:90:0.5);
        \end{tikzpicture}
        \caption{Configuration of QBPDs of $\sigma t_{ca}\gtrdot \sigma$}
        \label{fig:cfgsigca}
    \end{figure}
    \begin{proof}
        Since $\sigma t_{ca} \gtrdot\sigma$ we have $\sigma(c) < \sigma(a) = m$, so this is the configuration of the Rothe diagram of $\sigma t_{ca}$, where the pipe starting on row $a$ goes to column $\sigma(c) < m$ and the pipe starting on row $c < a$ goes to column $m$. There are no empty tiles and all 
        \begin{tikzpicture}[scale=\inlinescale]
\draw (0, 0) grid (1,1);
\draw[blue](0,0.5)--(1,0.5);
\end{tikzpicture} tiles or  \begin{tikzpicture}[scale=\inlinescale]
\draw (0, 0) grid (1,1);
\draw[blue](0,0.5) arc (90:0:0.5);
\end{tikzpicture} tiles are in the rightmost column, so by the same reasoning as in Proposition \ref{prop:configsig}, droop and lift moves cannot change this configuration. Thus, all QBPDs of $\sigma t_{ca}$ have this configuration.
    \end{proof}
\end{proposition}
\begin{proposition}
    In Setup \ref{su}, for $c > a$ such that $\sigma t_{ac} \lhd \sigma$, all QBPDs of $\sigma t_{ac}$ have the configuration as in Figure \ref{fig:sigac} when restricted to $[a, n]\times [m,  n]$.
    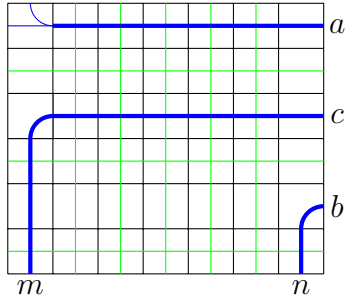
\begin{figure}[h!]
        \centering
        \begin{tikzpicture}[scale=\figscale]
            \node at (0.5, -0.3) {$m$};
            \node at (6.5, -0.3) {$n$};
            \node at (7.3, 5.5) {$a$};
            \node at (7.3, 1.5) {$b$};
            \node at (7.3, 3.5) {$c$};
            
            \draw (0, 0) grid (7,6);
            \draw[white] (1,7) --(2,7);
            \foreach \i in {1,2,3}
                {\draw[green] (\i.5, 0)--(\i.5, 6);
                \draw[green] (0, 1+\i.5)--(7,1+\i.5);
                } 
            \draw[green] (0, 0.5)--(7, 0.5);  
            \draw[green] (5.5, 0)--(5.5, 6);
            \draw[green] (4.5, 0)--(4.5, 6);
            \draw [blue, ultra thick] (1,3.5)--(7,3.5);
            \draw[blue, ultra thick] (1, 5.5)--(7, 5.5);
            \draw[fill=white] (0,3) rectangle (1,4);
            \draw[blue] (0.5, 6) arc (180:270:0.5);
            \draw[blue] (0, 5.5)--(1,5.5);
            \draw[blue, ultra thick] (0.5, 3) arc (180:90:0.5);
            \draw[blue, ultra thick] (0.5, 0)--(0.5, 3);
            \draw[blue, ultra thick] (6.5, 0)--(6.5, 1);
            \draw[blue, ultra thick] (6.5, 1) arc (180:90:0.5);
        \end{tikzpicture}
        \caption{Configuration for QBPDs of $\sigma t_{ac} \lhd \sigma$}
        \label{fig:sigac}
    \end{figure}
    \begin{proof}
        Since $\sigma t_{ac} \lhd \sigma$, we have $m  =\sigma(a) >\sigma(t) > \sigma(c)$ for all $a < t < c$ by Remark \ref{cond}, but $\sigma(b) > \sigma(a)$, so $c < b$. 
        Thus, the Rothe diagram of $\tau:= \sigma t_{ac}$ has the configuration as in Figure \ref{fig:sigac}, with the tile at $(a,m)$ being a \begin{tikzpicture}[scale=\inlinescale]
\draw (0, 0) grid (1,1);
\draw[blue](0,0.5)--(1,0.5);
\end{tikzpicture} tile. Droop moves cannot change this region since there are no empty tiles. For lift moves, the only  \begin{tikzpicture}[scale=\inlinescale]
\draw (0, 0) grid (1,1);
\draw[blue](0,0.5)--(1,0.5);
\end{tikzpicture} tiles or  \begin{tikzpicture}[scale=\inlinescale]
\draw (0, 0) grid (1,1);
\draw[blue](0,0.5) arc (90:0:0.5);
\end{tikzpicture} tiles not on the rightmost column are on column $m$. Now, for $a < y < c$ since $\tau(a) < \tau(t)$, the pipe at $(t, m)$ cannot be turned upward using a lift move since it would intersect with the pipe starting at row $a$. Thus, the only droop move that can change the region is to turn the tile at $(a,m)$ upward, which changes it into a \begin{tikzpicture}[scale=\inlinescale]
\draw (0, 0) grid (1,1);
\draw[blue](0.5,1) arc (180:270:0.5);
\end{tikzpicture} tile. After this, no more droop or lift moves can change the region.
    \end{proof}
\end{proposition}
We now analyze QBPDs for permutations of the form $\sigma t_{ca}\lhd \sigma, c  < a$. 
\begin{proposition}\label{ni}
    In Setup \ref{su}, we have the following properties:
    \begin{itemize}
        \item[(a)] If $S\neq \varnothing$ then $a_1 = a-1$.
        \item[(b)] $p_k > p_{k-1}> \dots > p_1 > p_0 = m$.
        \item[(c)] If $a_i < j < a_{i-1}$ for some $i$, then $m < \sigma(j) < p_{i-1}$.
        \item[(d)] If $a_i < j < a$ for some $i$ and $j\neq a_{i-1}$, then $m < \sigma(j) < p_{i-1}$.
        \item[(e)] If $a_i < j$ for some $i$ and $\sigma(j) > \sigma(a_{i-1})$, then $j = b$.
    \end{itemize}
    \begin{proof}
        For (a), if $S$ is not empty, suppose $s\in S$. Then by Remark $\ref{cond}$, $\sigma(s) > \sigma(t) > \sigma(a)$ for any $s < t < a$. In particular, if $s\neq a-1$, then $\sigma(a-1) > \sigma(a)$, which means $\sigma t_{a-1,a}\lhd \sigma$. Thus, $a-1\in S$ (and if $s = a-1$, we also have $a-1\in S$). By the ordering, $a_1 = a-1$.
        
        For (b), for each $1 \leq i \leq k$, note that $\sigma t_{a_i a} \lhd \sigma$. So, by Remark \ref{cond}, we have $\sigma(a_{i}) > \sigma(a)$ and $\sigma (a_i) > \sigma(t) >  \sigma(a)$ for all $a_i < t < a$. In particular, if $i > 1$, we have $\sigma(a_i) > \sigma(a_{i-1})$, which is $p_i > p_{i-1}$, and if $i = 1$, we have $\sigma(a_1) > \sigma(a)$, which is $p_1 > p_0  =m$.
        
        For (c), the statement is vacuously true when $i = 1$ since $a_1 = a_0  - 1$, so fix $i>1$ and $j$ such that $a_{i-1} < j < a_i$. We have $\sigma t_{a_i a} \lhd \sigma $, so by Remark \ref{cond}, we have $\sigma (a_i) > \sigma(t) >  \sigma(a)$ for all $a_i < t < a$. In particular, $\sigma(j) > \sigma(a) = m$. Now, suppose for contradiction that $\sigma(j) > \sigma(a_{i-1})$. Consider the set \begin{equation}X:= \{q < a_{i-1}: \sigma(q) > \sigma(a_{i-1})\}.\end{equation}
        Note that $j\in X$. Let $h$ be the maximum element of $X$. Then, $a_{i-1}> h\geq j > a_i$. We will show that $h\in S$, contradicting the fact that $S$ has no element strictly between $a_i$ and $a_{i-1}$ by the ordering. Consider any $r$ such that $h < r < a$. If $r > a_{i-1}$ we have $a_{i-1} <r < a$, and $\sigma t_{a_{i-1}a}\lhd \sigma $, so by Remark \ref{cond}, we have $\sigma(h) >\sigma(a_{i-1})>\sigma (r) > \sigma(a)$. If $r = a_{i-1}$ then $\sigma(h) > \sigma(a_{i-1}) > \sigma(a)$ by assumption. If $r < a_{i-1}$ then $h < r < a_{i-1}$, so $r\notin X$ (since $h$ is maximum of $X$). So, $\sigma(r) < \sigma(a_{i-1})$ and thus $\sigma(h) > \sigma(a_{i-1}) > \sigma(r) > \sigma(a)$. Thus, we showed that for any $h < r < a$,  we have $\sigma(h) > \sigma(r)> \sigma(a)$. Since we have $\sigma(h) > \sigma(a)$, Remark \ref{cond} gives $\sigma t_{ha}\lhd \sigma$, so $h\in S$.
        
        For (d), the statement is also vacuously true when $i = 1$, and when $i > 1$, by (c), we only need to consider the case $a_{i-1} < j < a$. However, we have $\sigma t_{a_{i-1}a}\lhd \sigma$. Then, by Remark \ref{cond}, we have $\sigma(a_{i-1}) > \sigma(j) > \sigma(a)$, which is $m < \sigma(j) < p_{i-1}$.
        
        For (e), by (d), we only need to consider $j\geq a$. We have $\sigma(j) > \sigma(a_{i-1})$ by assumption and $\sigma(a_{i-1}) \geq \sigma(a)$ by (b), so $j\neq a$ and thus $j > a$. However, recall that $b$ is the only number greater than $a$ such that $\sigma(b) > \sigma(a)$, so $j = b$.
    \end{proof}
\end{proposition}
\begin{proposition}
    In Setup \ref{su}, if $S\neq \varnothing$, any QBPD for $\sigma t_{a-1, a}$ has one of the two configurations in Figure \ref{fig:cfgxy}.
    \begin{figure}[h!]
        \centering
        \begin{tikzpicture}[scale=\figscale]
            \node at (0.5, -0.3) {$m$};
            \node at (6.5, -0.3) {$n$};
            \node at (2.5, -0.3) {$p_1$};
            \node at (3.5, -0.75) {$X_1$};
            \node at (7.3, 5.5) {$a$};
            \node at (7.85, 6.5) {$a-1$};
            \node at (7.3, 1.5) {$b$};
            \draw (0, 0) grid (7,7);
            \foreach \i in {1,2,3}
                {\draw[green] (2+\i.5, 0)--(2+\i.5, 7);
                \draw[green] (0, 1+\i.5)--(7,1+\i.5);
                } 
            \draw[green] (0, 0.5)--(7, 0.5);   
            \draw[green] (1.5, 0)--(1.5, 7);
            \draw[blue, ultra thick] (1, 6.5)--(7, 6.5);
            \draw [blue, ultra thick](1,6.5) arc (90:180:0.5);
            \draw[blue, ultra thick] (0.5, 0)--(0.5, 6);
            \draw[blue, ultra thick] (2.5, 0)--(2.5, 5);
            \draw[blue, ultra thick] (3, 5.5)--(7, 5.5);
            \draw[blue, ultra thick] (3, 5.5) arc(90:180:0.5);
            \draw[blue, ultra thick] (6.5, 0)--(6.5, 1);
            \draw[blue, ultra thick] (6.5, 1) arc (180:90:0.5);
        \end{tikzpicture}\begin{tikzpicture}[scale=\figscale]
            \node at (0.5, -0.3) {$m$};
            \node at (6.5, -0.3) {$n$};
            \node at (2.5, -0.3) {$p_1$};
            \node at (3.5, -0.75) {$Y_1$};
            \node at (7.3, 5.5) {$a$};
            \node at (7.85, 6.5) {$a-1$};
            \node at (7.3, 1.5) {$b$};
            \draw (0, 0) grid (7,7);
            \foreach \i in {1,2,3}
                {\draw[green] (2+\i.5, 0)--(2+\i.5, 7);
                \draw[green] (0, 1+\i.5)--(7,1+\i.5);
                } 
            \draw[green] (0, 0.5)--(7, 0.5);   
            \draw[green] (1.5, 0)--(1.5, 7);
            \draw[blue, ultra thick] (3, 6.5)--(7, 6.5);
            \draw [blue, ultra thick](3,6.5) arc (-90:-180:0.5);
            \draw[blue, ultra thick] (0.5, 0)--(0.5, 7);
            \draw[blue, ultra thick] (2.5, 0)--(2.5, 5);
            \draw[blue, ultra thick] (3, 5.5)--(7, 5.5);
            \draw[blue, ultra thick] (3, 5.5) arc(90:180:0.5);
            \draw[blue, ultra thick] (6.5, 0)--(6.5, 1);
            \draw[blue, ultra thick] (6.5, 1) arc (180:90:0.5);
        \end{tikzpicture}
        \caption{Configurations of QBPDs of $\sigma t_{a-1, a}$}
        \label{fig:cfgxy}
    \end{figure}
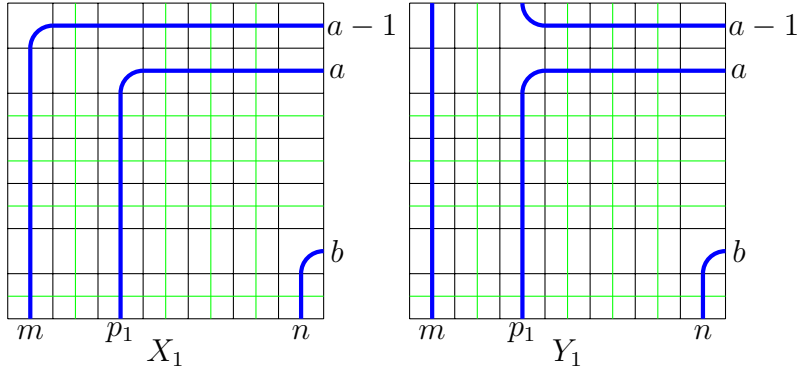
    \begin{proof}
        The Rothe diagram has configuration $X_1$. Droop moves cannot change the region, and the only tile that could possibly be turned upward by a lift move is the tile at $(a-1, p_1)$. If it got turned up by a lift move, we get configuration $Y_1$. When we reach configuration $Y_1$, no more droop or lift moves can change the configuration.
    \end{proof}
\end{proposition}
\begin{proposition}
    In Setup \ref{su}, for $i > 1$, any QBPD for $\sigma t_{a_ia}$ has one of the two configurations as in Figure \ref{fig:cfgfin}.
    \begin{figure}
        \centering
        \begin{tikzpicture}[scale=\figscale]
            \node at (6.5, -0.3) {$n$};
            \node at (2.5, -0.3) {$p_{i-1}$};
            \node at (4.5, -0.3) {$p_i$};
            \node at (4.5, -1.1) {$X_i$};
            \node at (7.32, 9.5) {$a_i$};
            \node at (7.3, 5.5) {$a$};
            \node at (7.6, 7.5) {$a_{i-1}$};
            \node at (7.3, 1.5) {$b$};
            \draw (2, 0) grid (7,8);
            \foreach \i in {1,2,3}
                {
                \draw[green] (2, 1+\i.5)--(7,1+\i.5);
                } 
            \draw[green] (2, 0.5)--(7, 0.5);   
            \draw[green] (2, 6.5)--(7, 6.5);
            \draw[blue, ultra thick] (3, 7.5)--(7, 7.5);
            \draw [blue, ultra thick](3,7.5) arc (90:180:0.5);
            \draw[blue, ultra thick] (2.5, 0)--(2.5, 7);
            \draw[blue, ultra thick] (4.5, 0)--(4.5, 5);
            \draw[blue, ultra thick] (5, 5.5)--(7, 5.5);
            \draw[blue, ultra thick] (5, 5.5) arc(90:180:0.5);
            \draw[blue, ultra thick] (6.5, 0)--(6.5, 1);
            \draw[blue, ultra thick] (6.5, 1) arc (180:90:0.5);
            \draw (2,8) grid (7,10);
            \foreach \i in {0, 1}{
                \draw[green] (3.5+\i+\i, 0)--(3.5+\i+\i, 10);
            }
            \draw[green] (2,8.5)--(7,8.5);
            \draw[blue, ultra thick] (3, 9.5)--(7, 9.5);
            \draw[blue] (2.5, 10) arc (180:270:0.5);
            \draw[blue] (2, 9.5)--(3,9.5);
        \end{tikzpicture}
        \begin{tikzpicture}[scale=\figscale]
            \node at (6.5, -0.3) {$n$};
            \node at (2.5, -0.3) {$p_{i-1}$};
            \node at (4.5, -0.3) {$p_i$};
            \node at (4.5, -1.1) {$Y_i$};
            \node at (7.32, 9.5) {$a_i$};
            \node at (7.3, 5.5) {$a$};
            \node at (7.6, 7.5) {$a_{i-1}$};
            \node at (7.3, 1.5) {$b$};
            \draw (2, 0) grid (7,8);
            \foreach \i in {1,2,3}
                {
                \draw[green] (2, 1+\i.5)--(7,1+\i.5);
                } 
            \draw[green] (3.5, 0)--(3.5, 8);
            \draw[green] (2, 0.5)--(7, 0.5);   
            \draw[green] (2, 6.5)--(7, 6.5);
            \draw[blue, ultra thick] (3, 7.5)--(7, 7.5);
            \draw [blue, ultra thick](3,7.5) arc (90:180:0.5);
            \draw[blue, ultra thick] (2.5, 0)--(2.5, 7);
            \draw[blue, ultra thick] (4.5, 0)--(4.5, 5);
            \draw[blue, ultra thick] (5, 5.5)--(7, 5.5);
            \draw[blue, ultra thick] (5, 5.5) arc(90:180:0.5);
            \draw[blue, ultra thick] (6.5, 0)--(6.5, 1);
            \draw[blue, ultra thick] (6.5, 1) arc (180:90:0.5);
            \draw (4,8) grid (7,10);
            \draw[green] (5.5, 0)--(5.5, 10);
            \draw[green] (4,8.5)--(7,8.5);
            \draw[blue, ultra thick] (5, 9.5)--(7, 9.5);
            \draw[blue, ultra thick] (4.5, 10) arc (180:270:0.5);
        \end{tikzpicture}
        \caption{Configurations of QBPDs of $\sigma t_{a_ia} \lhd \sigma, i > 1$}
        \label{fig:cfgfin}
    \end{figure}
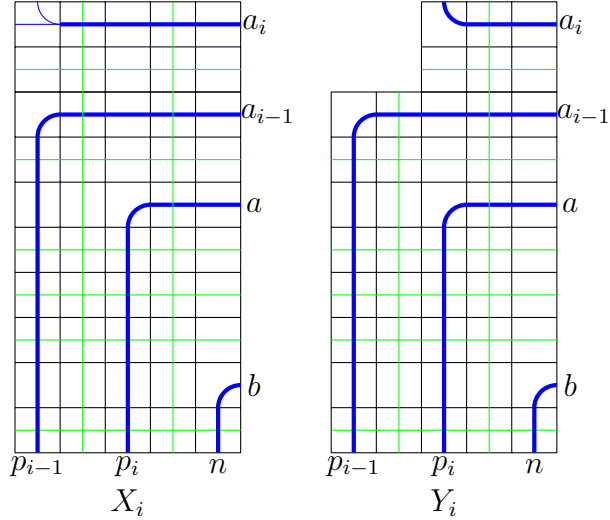
    \begin{proof}
         Let $\tau = \sigma t_{a_{i}a}$. If $t > a_i$ and $\sigma(t) >\tau(a_{i-1}) = p_{i-1}$, then $t = b$ by Proposition \ref{ni} (e). Since $\tau(t) = \sigma(t)$ except when $t\in \{a, a_i\}$ and $\tau(a) = \sigma(a_i) = p_i > p_{i-1}$, we have if $t > a_i$ and $\tau(t) > p_{i-1}$, then $t\in \{a, b\}$. Thus, the Rothe diagram for $\tau$ has configuration $X_i$ with the tile at $(a_i, p_{i-1})$ being a \begin{tikzpicture}[scale=\inlinescale]
\draw (0, 0) grid (1,1);
\draw[blue](0,0.5)--(1,0.5);
\end{tikzpicture} tile. Now, no droop moves could affect the $[a_i, n]\times [p_{i-1}, n]$ region. For lift moves, the only tiles that could be turned up are \begin{tikzpicture}[scale=\inlinescale]
\draw (0, 0) grid (1,1);
\draw[blue](0,0.5)--(1,0.5);
\end{tikzpicture} tiles on column $p_{i-1}$ and $p_i$. For $a_i < t < a$, since $\tau(t) = \sigma(t) > \sigma(a) = \tau(a_i)$, the pipe starting on row $t$ cannot intersect with the pipe starting on row $a_i$, so we cannot use a lift move on \begin{tikzpicture}[scale=\inlinescale]
\draw (0, 0) grid (1,1);
\draw[blue](0,0.5)--(1,0.5);
\end{tikzpicture} tiles at $(t, p_{i-1})$ or $ (t, p_i)$. Thus, the only possible tiles that can be turned up using lift moves are at $(a_i, p_{i-1})$ or $(a_i, p_{i})$, which gives the configuration $X_i$ with $(a_i, p_{i-1})$ being a \begin{tikzpicture}[scale=\inlinescale]
\draw (0, 0) grid (1,1);
\draw[blue](0.5,1) arc (180:270:0.5);
\end{tikzpicture} tile or configuration $Y_i$. After this, the tiles shown in configuration $Y_i$ cannot be further changed. The configuration $X_i$ with $(a_i, p_{i-1})$ being a \begin{tikzpicture}[scale=\inlinescale]
\draw (0, 0) grid (1,1);
\draw[blue](0.5,1) arc (180:270:0.5);
\end{tikzpicture} tile also cannot be changed further.
    \end{proof}
\end{proposition}
We now start giving bijections between QBPDs and give the relationships between their $\bwt$'s (as defined in Definition \ref{def:bwt}). Define pipe $i$ to be the pipe starting at row $i$. Write $\QBPD(w, H)$ for the set of QBPDs for $w$ with configuration $H$. For example, $\QBPD(\pi, A)$ denotes the set of QBPDs of $\pi$ with configuration $A$.
\begin{lemma}\label{lem:phia}
    In Setup \ref{su}, there is a bijection $\phi_A\colon \QBPD(\pi, A)\to \QBPD(\sigma)$ such that $(x_a-y_m)\bwt(\phi_A(P)) = \bwt(P)$ for all $P\in \QBPD(\pi , A)$. 
    \begin{proof}
        Recall that $\pi = \sigma t_{ab}$ and $m = \sigma(a) = \pi(b)$. Also recall Configuration $A$ from Proposition \ref{prop:configpi}, and recall Proposition \ref{prop:configsig}. The bijection $\phi_A$ is formed by making the change in Figure \ref{phia}.
        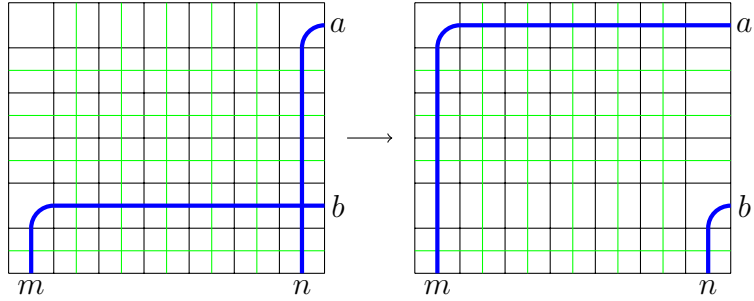
\begin{figure}[h!]
            \centering
            \begin{tikzpicture}[scale=\figscale]
                \node at (0.5, -0.3) {$m$};
                \node at (6.5, -0.3) {$n$};
                \node at (7.3, 5.5) {$a$};
                \node at (7.3, 1.5) {$b$};
                \draw (0, 0) grid (7,6);
                \foreach \i in {1,2,3}
                    {\draw[green] (\i.5, 0)--(\i.5, 6);
                    \draw[green] (0, 1+\i.5)--(7,1+\i.5);
                    } 
                \draw[green] (0, 0.5)--(7, 0.5);   
                \draw[green] (5.5, 0)--(5.5, 6);
                \draw[green] (4.5, 0)--(4.5, 6);
                \draw[blue, ultra thick] (1, 1.5)--(7, 1.5);
                \draw [blue, ultra thick](1,1.5) arc (90:180:0.5);
                \draw[blue, ultra thick] (0.5, 0)--(0.5, 1);
                \draw[blue, ultra thick] (6.5, 0)--(6.5, 5);
                \draw[blue, ultra thick] (6.5, 5) arc (180:90:0.5);
                \draw[->] (7.5, 3)--(8.5, 3);
                \begin{scope}[shift={(9,0)}]
                    
                        \node at (0.5, -0.3) {$m$};
                        \node at (6.5, -0.3) {$n$};
                        \node at (7.3, 5.5) {$a$};
                        \node at (7.3, 1.5) {$b$};
                        \draw (0, 0) grid (7,6);
                        \foreach \i in {1,2,3}
                            {\draw[green] (\i.5, 0)--(\i.5, 6);
                            \draw[green] (0, 1+\i.5)--(7,1+\i.5);
                            } 
                        \draw[green] (0, 0.5)--(7, 0.5);   
                        \draw[green] (5.5, 0)--(5.5, 6);
                        \draw[green] (4.5, 0)--(4.5, 6);
                        \draw[blue, ultra thick] (1, 5.5)--(7, 5.5);
                        \draw [blue, ultra thick](1,5.5) arc (90:180:0.5);
                        \draw[blue, ultra thick] (0.5, 0)--(0.5, 5);
                        \draw[blue, ultra thick] (6.5, 0)--(6.5, 1);
                        \draw[blue, ultra thick] (6.5, 1) arc (180:90:0.5);
                    
                \end{scope}
            \end{tikzpicture}
            \caption{The bijection $\phi_A$}\label{phia}
        \end{figure}
        All the changed pipes are completely in the changed region and one can see that they do not cross any pipe twice on both sides of the bijection. To reverse the map, we simply need to make the reverse change, so this is a bijection. The LHS has one more empty cell at $(a,m)$ compared to the RHS, so we have $(x_a-y_m)\bwt(\phi_A(P)) = \bwt(P)$ for all $P\in \QBPD(\pi , A)$.
    \end{proof}
\end{lemma}
\begin{lemma}\label{lem:phib}
    In Setup \ref{su}, there is a bijection $\phi_B\colon \QBPD(\pi, B)\to \bigcup\limits_{\substack{c < a\\\sigma t_{ca} \gtrdot \sigma}}\QBPD(\sigma t_{ca})$ such that $\bwt(\phi_B(P)) = \bwt(P)$ for all $P\in \QBPD(\pi, B)$.
    \begin{proof}
        The bijection $\phi_B$ is formed by making the change in Figure \ref{fig:phib}. Let $c$ be the row at which the pipe of the elbow at $(a, m)$ on the LHS of the bijection starts. Then $c < a$, and if $c < t < a$, then pipe $t$ must either pass through a cell in column $m$ that is above the $(a, m)$ elbow or eventually become one of the vertical green pipes. If it passes through a cell in column $m$ that is above the $(a, m)$ elbow, it gets above  pipe $c$, but it started below pipe $c$, so $\sigma(t) < \sigma(c)$. Otherwise, if it eventually becomes one of the vertical green pipes, then $\sigma(t) > \sigma(a)$. Thus, $\sigma t_{ca} \gtrdot\sigma$. The new pipe $a$ in the RHS crosses some vertical green pipe before reaching $(a,m)$ and after this, it cannot cross these vertical green pipes again since pipes do not move right, and after $(a,m)$, it took the path of the old pipe $c$ which will not cross any pipes twice. By similar reasoning, the new pipe $c$ (which ends up in column $m$) also does not cross any other pipe twice, so the QBPD on the RHS is a reduced QBPD for $\sigma t_{ca}$. 
        
        We show that this a bijection by showing the reverse operation: given a QBPD of $\tau = \sigma t_{ca}\gtrdot \sigma$ for some $c < a$, we do the reverse change from the RHS to the LHS (this is well defined by Proposition ~\ref{prop:sigcagtr}). Note that, for a QBPD of $\tau$, the pipe ending at column $m$ starts at row $c$, and no pipe starting at row $t < a$ and ending at column $\tau(t) < m$ simultaneously crosses pipes $a$ and $c$ since otherwise we have $a < t < c$ and $\tau(a) < \tau(t) < \tau(c)$, contradicting $\tau = \sigma t_{ca}\gtrdot \sigma$. So, doing the reverse change will not make these pipes intersect with pipe $c$ twice. Any pipe that starts at row $t > a$ or ends at column $\tau(t) > m$ is a green pipe in the configuration and doing the reverse change from RHS to LHS cannot make any green pipe intersect with pipe $c$ twice. Thus, doing the reverse change from RHS to LHS gives a reduced QBPD for $\pi$ since all the changed pipes do not cross any other pipe twice. The LHS and the RHS have the same set of tiles that contribute weight, so the weights are the same.
        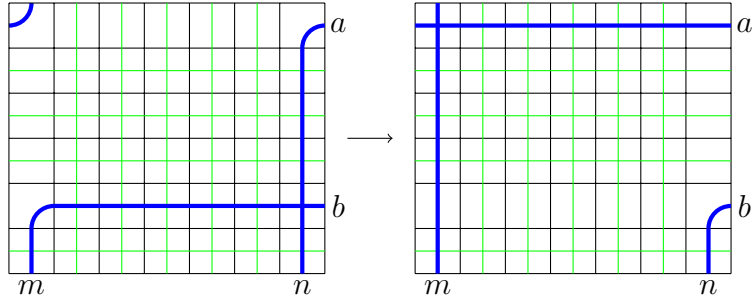
\begin{figure}[h!]
            \centering
            \begin{tikzpicture}[scale=\figscale]
                \node at (0.5, -0.3) {$m$};
            \node at (6.5, -0.3) {$n$};
            \node at (7.3, 5.5) {$a$};
            \node at (7.3, 1.5) {$b$};
            \draw[blue, ultra thick] (0.5, 6) arc (0:-90:0.5);
            \draw (0, 0) grid (7,6);
            \foreach \i in {1,2,3}
                {\draw[green] (\i.5, 0)--(\i.5, 6);
                \draw[green] (0, 1+\i.5)--(7,1+\i.5);
                } 
            \draw[green] (0, 0.5)--(7, 0.5);
            \draw[green] (5.5, 0)--(5.5, 6);
            \draw[green] (4.5, 0)--(4.5, 6);
            \draw[blue, ultra thick] (1, 1.5)--(7, 1.5);
            \draw [blue, ultra thick](1,1.5) arc (90:180:0.5);
            \draw[blue, ultra thick] (0.5, 0)--(0.5, 1);
            \draw[blue, ultra thick] (6.5, 0)--(6.5, 5);
            \draw[blue, ultra thick] (6.5, 5) arc (180:90:0.5);
            \draw[->] (7.5, 3)--(8.5, 3);
                \begin{scope}[shift={(9,0)}]
                    
                    \node at (0.5, -0.3) {$m$};
                    \node at (6.5, -0.3) {$n$};
                    \node at (7.3, 5.5) {$a$};
                    \node at (7.3, 1.5) {$b$};
                    \draw (0, 0) grid (7,6);
                    \foreach \i in {1,2,3}
                        {\draw[green] (\i.5, 0)--(\i.5, 6);
                        \draw[green] (0, 1+\i.5)--(7,1+\i.5);
                        } 
                    \draw[green] (0, 0.5)--(7, 0.5);   
                    \draw[green] (5.5, 0)--(5.5, 6);
                    \draw[green] (4.5, 0)--(4.5, 6);
                    \draw[blue, ultra thick] (0, 5.5)--(7, 5.5);
                    \draw[blue, ultra thick] (0.5, 0)--(0.5, 6);
                    \draw[blue, ultra thick] (6.5, 0)--(6.5, 1);
                    \draw[blue, ultra thick] (6.5, 1) arc (180:90:0.5);
                    
                \end{scope}
            \end{tikzpicture}
            \caption{The bijection $\phi_B$}
            \label{fig:phib}
        \end{figure}
    \end{proof}
\end{lemma}
\begin{lemma}\label{lem:phic}
    In Setup \ref{su}, there is a bijection $\phi_C\colon \QBPD(\pi, C)\to \bigcup\limits_{\substack{c > a\\\sigma t_{ac} \lhd \sigma}}\QBPD(\sigma t_{ac})$ such that $-q_{ac}\bwt(\phi_C(P)) = \bwt(P)$ for all $P\in \QBPD(\pi, C)$. 
    \begin{proof}
        
        \begin{figure}[h!]
            \centering
            \begin{tikzpicture}[scale=\figscale]
                \node at (0.5, -0.3) {$m$};
            \node at (6.5, -0.3) {$n$};
            \node at (7.3, 5.5) {$a$};
            \node at (7.3, 1.5) {$b$};
            \node at (7.3, 3.5) {$c$};
            
            \draw (0, 0) grid (7,6);
            \draw[white] (1,7) --(2,7);
            \foreach \i in {1,2,3}
                {\draw[green] (\i.5, 0)--(\i.5, 6);
                \draw[green] (0, 1+\i.5)--(7,1+\i.5);
                } 
            \draw[green] (0, 0.5)--(7, 0.5);  
            \draw[green] (5.5, 0)--(5.5, 6);
            \draw[green] (4.5, 0)--(4.5, 6);
            \draw [blue, ultra thick] (0.5,4)--(0.5,5);
            \draw [blue, ultra thick] (1,3.5)--(7,3.5);
            \draw[blue, ultra thick] (1, 1.5)--(7, 1.5);
            \draw[fill=white] (0,3) rectangle (1,4);
            \draw[blue] (0.5, 5) arc (0:90:0.5);
            \draw[blue] (0.5, 5)--(0.5,6);
            \draw[blue, ultra thick] (0.5, 4) arc (180:270:0.5);
            \draw [blue, ultra thick](1,1.5) arc (90:180:0.5);
            \draw[blue, ultra thick] (0.5, 0)--(0.5, 1);
            \draw[blue, ultra thick] (6.5, 0)--(6.5, 5);
            \draw[blue, ultra thick] (6.5, 5) arc (180:90:0.5);
            \draw[->] (7.5, 3)--(8.5, 3);
                \begin{scope}[shift={(9,0)}]
                    
                    \node at (0.5, -0.3) {$m$};
            \node at (6.5, -0.3) {$n$};
            \node at (7.3, 5.5) {$a$};
            \node at (7.3, 1.5) {$b$};
            \node at (7.3, 3.5) {$c$};
            
            \draw (0, 0) grid (7,6);
            \draw[white] (1,7) --(2,7);
            \foreach \i in {1,2,3}
                {\draw[green] (\i.5, 0)--(\i.5, 6);
                \draw[green] (0, 1+\i.5)--(7,1+\i.5);
                } 
            \draw[green] (0, 0.5)--(7, 0.5);  
            \draw[green] (5.5, 0)--(5.5, 6);
            \draw[green] (4.5, 0)--(4.5, 6);
            \draw [blue, ultra thick] (1,3.5)--(7,3.5);
            \draw[blue, ultra thick] (1, 5.5)--(7, 5.5);
            \draw[fill=white] (0,3) rectangle (1,4);
            \draw[blue] (0.5, 6) arc (180:270:0.5);
            \draw[blue] (0, 5.5)--(1,5.5);
            \draw[blue, ultra thick] (0.5, 3) arc (180:90:0.5);
            \draw[blue, ultra thick] (0.5, 0)--(0.5, 3);
            \draw[blue, ultra thick] (6.5, 0)--(6.5, 1);
            \draw[blue, ultra thick] (6.5, 1) arc (180:90:0.5);
                    
                \end{scope}
            \end{tikzpicture}
            \caption{The bijection $\phi_C$}
            \label{fig:phic}
        \end{figure}
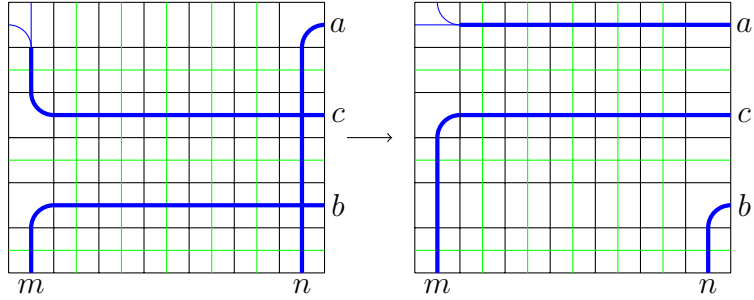
        The bijection $\phi_C$ is formed by making the change in Figure \ref{fig:phic}. When we make the change from the LHS to the RHS, note that the new pipe $a$ (pipe $a$ on the RHS) crosses the same pipes as old pipe $c$ in the same column, except for pipe $t$ with $a < t < c$. Thus, if the QBPD on the LHS is reduced, then the one on the RHS is reduced since all changed pipes cannot intersect with any other pipe twice. For the inverse map, given a reduced QBPD of $\tau =\sigma t_{ac} \lhd \sigma$, we have for all $a < t < c$, $\tau(a) < \tau(t) < \tau(c)$ by Remark \ref{cond}, so pipe $c$ and pipe $a$ do not intersect with pipe $t$. Then, doing the reverse change does not make the new pipe $c$ intersect with any other pipe twice, and the new pipe $a$ and $b$ also do not intersect with any other pipe twice, so doing the reverse change gives a reduced QBPD for $\pi$. The LHS has a \begin{tikzpicture}[scale=\inlinescale]
\draw (0, 0) grid (1,1);
\draw[blue](0.5,0) arc (0:90:0.5);
\end{tikzpicture} or \begin{tikzpicture}[scale=\inlinescale]
\draw (0, 0) grid (1,1);
\draw[blue](0.5,0)--(0.5,1);
\end{tikzpicture} in which the pipe moves upward at $(a, m)$ contributing $-q_a$ and, for $a < t <c$, there is a \begin{tikzpicture}[scale=\inlinescale]
\draw (0, 0) grid (1,1);
\draw[blue](0.5,0)--(0.5,1);
\draw[blue](0,0.5)--(1,0.5);
\end{tikzpicture} tile at $(t, m)$ in which the vertical pipe moves upward, which contributes $q_t$. These tiles are not present in the RHS. Thus, the weight is multiplied by $-q_{ac}$.
    \end{proof}
\end{lemma}
\begin{lemma}\label{lem:phid}
    In Setup \ref{su}, there is a bijection \begin{equation}\phi_D\colon \QBPD(\pi, D)\to \QBPD(\sigma t_{a-1,a},X_1)_{\sigma t_{a-1,a}\lhd \sigma}\end{equation} such that $q_{a-1}\bwt(\phi_D(P)) = \bwt(P)$ for all $P\in \QBPD(\pi, D)$. Here, the notation \\$\QBPD(\sigma t_{a-1,a},X_1)_{\sigma t_{a-1,a}\lhd \sigma}$ means the set $\QBPD(\sigma t_{a-1,a}, X_1)$ if $\sigma t_{a-1,a}\lhd \sigma$ and the empty set otherwise. (If the set is empty, this will be the trivial bijection between two empty sets.) 
    \begin{proof}
        
        \begin{figure}[h!]
            \centering
            \begin{tikzpicture}[scale=\figscale]
                \node at (0.5, -0.3) {$m$};
            \node at (6.5, -0.3) {$n$};
            \node at (2.5, -0.3) {$p_1$};
            \node at (7.3, 5.5) {$a$};
            \node at (7.8, 6.5) {$a-1$};
            \node at (7.3, 1.5) {$b$};
            \draw (0, 0) grid (7,7);
            \draw[fill=white] (0,5) rectangle (1,7);
            \foreach \i in {1,2,3}
                {\draw[green] (2+\i.5, 0)--(2+\i.5, 7);
                \draw[green] (0, 1+\i.5)--(7,1+\i.5);
                } 
            \draw[green] (0, 0.5)--(7, 0.5);  
            \draw[green] (1.5, 0)--(1.5, 7);
            \draw[blue, ultra thick] (2.5, 0)--(2.5, 6);
            \draw[blue, ultra thick] (3, 6.5)--(7, 6.5);
            \draw[blue, ultra thick] (3, 6.5) arc (90:180:0.5);
            \draw[blue, ultra thick] (1, 1.5)--(7, 1.5);
            \draw [blue, ultra thick](1,1.5) arc (90:180:0.5);
            \draw[blue, ultra thick] (0.5, 0)--(0.5, 1);
            \draw[blue, ultra thick] (6.5, 0)--(6.5, 5);
            \draw[blue, ultra thick] (6.5, 5) arc (180:90:0.5);
            \draw[->] (7.5, 3)--(8.5, 3);
                \begin{scope}[shift={(9,0)}]
                    \node at (0.5, -0.3) {$m$};
                    \node at (6.5, -0.3) {$n$};
                    \node at (2.5, -0.3) {$p_1$};
                    \node at (7.3, 5.5) {$a$};
                    \node at (7.8, 6.5) {$a-1$};
                    \node at (7.3, 1.5) {$b$};
                    \draw (0, 0) grid (7,7);
                    \foreach \i in {1,2,3}
                        {\draw[green] (2+\i.5, 0)--(2+\i.5, 7);
                        \draw[green] (0, 1+\i.5)--(7,1+\i.5);
                        } 
                    \draw[green] (0, 0.5)--(7, 0.5);   
                    \draw[green] (1.5, 0)--(1.5, 7);
                    \draw[blue, ultra thick] (1, 6.5)--(7, 6.5);
                    \draw [blue, ultra thick](1,6.5) arc (90:180:0.5);
                    \draw[blue, ultra thick] (0.5, 0)--(0.5, 6);
                    \draw[blue, ultra thick] (2.5, 0)--(2.5, 5);
                    \draw[blue, ultra thick] (3, 5.5)--(7, 5.5);
                    \draw[blue, ultra thick] (3, 5.5) arc(90:180:0.5);
                    \draw[blue, ultra thick] (6.5, 0)--(6.5, 1);
                    \draw[blue, ultra thick] (6.5, 1) arc (180:90:0.5);
                \end{scope}
            \end{tikzpicture}
            \caption{The bijection $\phi_D$}
            \label{fig:phid}
        \end{figure}
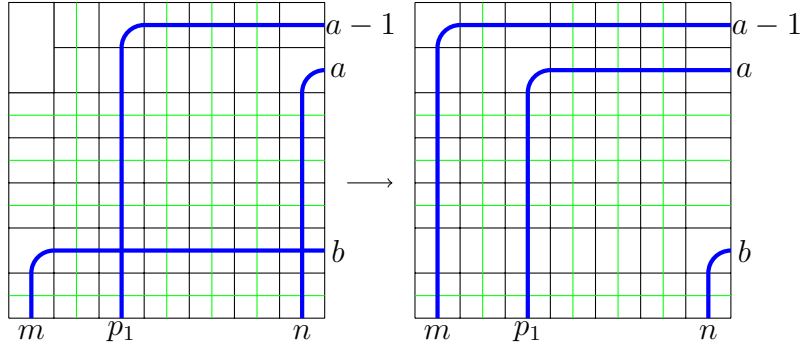
        The bijection $\phi_D$ is formed by making the change in Figure \ref{fig:phid}. We first show that if there exists a QBPD of configuration $D$, as in Figure \ref{fig:configpi}, then we can deduce that it has the configuration as in the LHS of Figure \ref{fig:phid} and that $\sigma{t_{a-1, a}}\lhd \sigma$. Let $(a-1, p)$ be the rightmost cell in row $a-1$ that is not 
        \begin{tikzpicture}[scale=\inlinescale]
            \draw (0, 0) rectangle (1,1);
            \draw[blue] (0, 0.5)--(1,0.5);
        \end{tikzpicture} 
        or 
        \begin{tikzpicture}[scale=\inlinescale]
            \draw (0, 0) rectangle (1,1);
            \draw[blue] (0, 0.5)--(1,0.5);
            \draw[blue] (0.5, 0)--(0.5,1);
        \end{tikzpicture}. Note that $p < n$ since, at the rightmost column, pipe $a-1$ cannot move up (Remark \ref{uplast}) or down (due to the elbow right below), and $p > m$ since $(a-1, m)$ is occupied by a domino. Then, for $p < t < n$, the pipe $a-1$ moves horizontally through $(a-1, t)$ and the green pipes coming from below must cross with it, so these tiles must be \begin{tikzpicture}[scale=\inlinescale]
            \draw (0, 0) rectangle (1,1);
            \draw[blue] (0, 0.5)--(1,0.5);
            \draw[blue] (0.5, 0)--(0.5,1);
        \end{tikzpicture} tiles. Now, at $(a-1, p)$, we must have the \begin{tikzpicture}[scale=\inlinescale]
        \draw (0, 0) grid (1,1);
        \draw[blue](1,0.5) arc (90:180:0.5);
        \end{tikzpicture} elbow since it is not 
        \begin{tikzpicture}[scale=\inlinescale]
            \draw (0, 0) rectangle (1,1);
            \draw[blue] (0, 0.5)--(1,0.5);
            \draw[blue] (0.5, 0)--(0.5,1);
        \end{tikzpicture} and $(a,p)$ is a 
        \begin{tikzpicture}[scale=\inlinescale]
            \draw (0, 0) rectangle (1,1);
            
            \draw[blue] (0.5, 0)--(0.5,1);
        \end{tikzpicture} tile. For $m < t < p$, the green pipes below must pass vertically through $(a-1,t)$ since the rightmost pipe that turns with a \begin{tikzpicture}[scale=\inlinescale]
        \draw (0, 0) grid (1,1);
        \draw[blue](1,0.5) arc (90:180:0.5);
        \end{tikzpicture} elbow would meet the elbow at $(a-1, p)$. (Note that turning with \begin{tikzpicture}[scale=\inlinescale]
        \draw (0, 0) grid (1,1);
        \draw[blue](0,0.5) arc (90:0:0.5);
        \end{tikzpicture} elbow is not possible since the pipe would be moving rightward---recall that pipes move from the right edge to the bottom edge.)
        Thus, we have the configuration as in the LHS.
        
        Note that this means $p = p_1 = \sigma(a-1) > m = \sigma(a)$, so $\sigma t_{a-1, a}\lhd \sigma$. Thus, if $\sigma t_{a-1, a}\ntriangleleft \sigma$, this means there are no QBPDs of $\pi$ with configuration $D$, so the bijection would just be the trivial bijection between two empty sets.
        
        Now, consider the case when there exists QBPDs of $\pi$ with configuration $D$. In this case, the changed pipes all lie entirely in the rectangle, so we can see directly from Figure \ref{fig:phid} that starting with a reduced QBPD on the left results in a reduced QBPD on the right and vice versa. Note that the LHS has an extra domino compared to the RHS, so we have $q_{a-1}\bwt(\phi_D(P)) = \bwt(P)$ for all QPBDs $P\in \QBPD(\pi, D)$.
    \end{proof}
\end{lemma}
\begin{lemma}\label{lem:phii}
    In Setup \ref{su}, for each  $i > 1$, there is a bijection $\phi_i\colon \QBPD(\sigma t_{a_ia}, X_i)\to \QBPD(\sigma t_{a_{i-1} a}, Y_{i-1})$ such that $\bwt(\phi_i(P)) = -q_{a_ia_{i-1}}\bwt(P)$ for all $P\in \QBPD(\sigma t_{a_ia}, X_i)$. In particular, since there is no $X_{k+1}$, the set $\QBPD(\sigma t_{a_ka}, Y_k)$ is empty. 
    \begin{proof}
        
        \begin{figure}[h!]
            \centering
            \begin{tikzpicture}[scale=\figscale]
                \node at (6.5, -0.3) {$n$};
                \node at (2.5, -0.3) {$p_{i-1}$};
                \node at (4.5, -0.3) {$p_i$};
                \node at (7.32, 9.5) {$a_i$};
                \node at (7.3, 5.5) {$a$};
                \node at (7.6, 7.5) {$a_{i-1}$};
                \node at (7.3, 1.5) {$b$};
                \draw (2, 0) grid (7,8);
                \foreach \i in {1,2,3}
                    {
                    \draw[green] (2, 1+\i.5)--(7,1+\i.5);
                    } 
                \draw[green] (2, 0.5)--(7, 0.5);   
                \draw[green] (2, 6.5)--(7, 6.5);
                \draw[blue, ultra thick] (3, 7.5)--(7, 7.5);
                \draw [blue, ultra thick](3,7.5) arc (90:180:0.5);
                \draw[blue, ultra thick] (2.5, 0)--(2.5, 7);
                \draw[blue, ultra thick] (4.5, 0)--(4.5, 5);
                \draw[blue, ultra thick] (5, 5.5)--(7, 5.5);
                \draw[blue, ultra thick] (5, 5.5) arc(90:180:0.5);
                \draw[blue, ultra thick] (6.5, 0)--(6.5, 1);
                \draw[blue, ultra thick] (6.5, 1) arc (180:90:0.5);
                \draw (2,8) grid (7,10);
                \foreach \i in {0, 1}{
                    \draw[green] (3.5+\i+\i, 0)--(3.5+\i+\i, 10);
                }
                \draw[green] (2,8.5)--(7,8.5);
                \draw[blue, ultra thick] (3, 9.5)--(7, 9.5);
                \draw[blue] (2.5, 10) arc (180:270:0.5);
                \draw[blue] (2, 9.5)--(3,9.5);
            \draw[->] (7.5, 5)--(8.5, 5);
                \begin{scope}[shift={(7,0)}]
                    \node at (6.5, -0.3) {$n$};
                \node at (2.5, -0.3) {$p_{i-1}$};
                \node at (4.5, -0.3) {$p_i$};
                \node at (7.32, 9.5) {$a_i$};
                \node at (7.3, 5.5) {$a$};
                \node at (7.6, 7.5) {$a_{i-1}$};
                \node at (7.3, 1.5) {$b$};
                \draw (2, 0) grid (7,8);
                \foreach \i in {1,2,3}
                    {
                    \draw[green] (2, 1+\i.5)--(7,1+\i.5);
                    } 
                \draw[green] (2, 0.5)--(7, 0.5);   
                \draw[green] (2, 6.5)--(7, 6.5);
                \draw[blue, ultra thick] (2.5, 8)--(2.5, 9);
                \draw[blue, ultra thick] (3, 7.5)--(7, 7.5);
                \draw [blue, ultra thick](3,7.5) arc (-90:-180:0.5);
                \draw[blue, ultra thick] (2.5, 0)--(2.5, 5);
                \draw[blue, ultra thick] (4.5, 0)--(4.5, 9);
                \draw[blue, ultra thick] (3, 5.5)--(7, 5.5);
                \draw[blue, ultra thick] (3, 5.5) arc(90:180:0.5);
                \draw[blue, ultra thick] (6.5, 0)--(6.5, 1);
                \draw[blue, ultra thick] (6.5, 1) arc (180:90:0.5);
                \draw (2,8) grid (7,10);
                \foreach \i in {0, 1}{
                    \draw[green] (3.5+\i+\i, 0)--(3.5+\i+\i, 10);
                }
                \draw[green] (2,8.5)--(7,8.5);
                \draw[blue, ultra thick] (5, 9.5)--(7, 9.5);
                \draw[blue] (2.5, 9) arc (0:90:0.5);
                \draw[blue, ultra thick] (4.5, 9) arc (180:90:0.5);
                \draw[blue] (2.5, 9)--(2.5,10);
                \end{scope}
            \end{tikzpicture}
            \caption{The bijection $\phi_i$}
            \label{fig:phii}
        \end{figure}
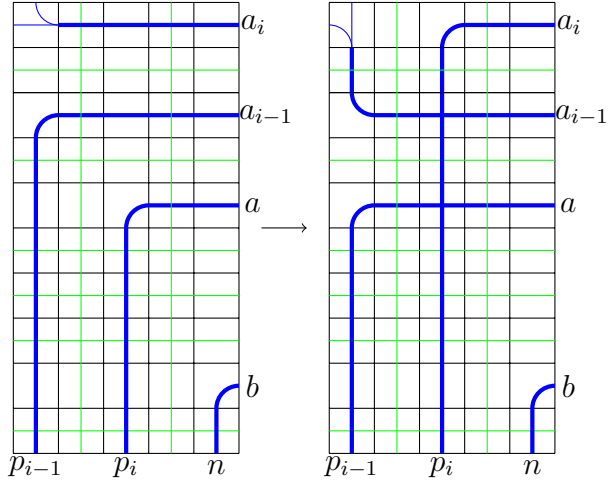
        The bijection $\phi_i$ is formed by making the change in Figure \ref{fig:phii}. We first claim that all QBPDs of $\sigma t_{a_{i-1}a}$ that have configuration 
        $Y_{i-1}$ actually have the configuration as in the RHS of Figure \ref{fig:phii}. Note that configuration $Y_{i-1}$ and the RHS of Figure \ref{fig:phii} agree on the region $[a_{i-1}, n]\times [p_{i-1}, n]$ (whether $i-1 = 1$ or $i-1 > 1$, see Figure \ref{fig:cfgfin} and Figure \ref{fig:cfgxy}). Now, for a QBPD that has configuration $Y_{i-1}$, let $(c, p_{i-1})$ be the first tile that is not a 
        \begin{tikzpicture}[scale=\inlinescale]
\draw (0, 0) grid (1,1);
\draw[blue](0.5,0)--(0.5,1);
\draw[blue](0,0.5)--(1,0.5);
\end{tikzpicture} 
when tracing the $a_{i-1}$ pipe upward from $(a_{i-1}, p_{i-1})$. Then, for the rows $j$ with $c < j < a_{i-1}$, there are pipes moving horizontally through $(j, p_{i-1})$, and for $p_{i-1} < t < n$, there are pipes moving vertically through $(a_{i-1}, t)$. All of these horizontal and vertical pipes are pairwise distinct since pipes do not move right, so they must all cross each other. So, the region $[c+1, a_{i-1}-1]\times [p_{i-1}+1, n-1]$ must be filled with \begin{tikzpicture}[scale=\inlinescale]
\draw (0, 0) grid (1,1);
\draw[blue](0.5,0)--(0.5,1);
\draw[blue](0,0.5)--(1,0.5);
\end{tikzpicture} tiles. 
This means the region $[c+1, a_{i-1}-1]\times \{n\}$ must be filled with \begin{tikzpicture}[scale=\inlinescale]
\draw (0, 0) grid (1,1);
\draw[blue](0,0.5)--(1,0.5);
\end{tikzpicture} tiles. Now, for row $c$, let $(p, c)$ be the first tile where pipe $c$ does not move horizontally. Note that $p < n$ since, at $(c, n)$, the pipe cannot move upward (Remark \ref{uplast}) and it cannot move downward since it would meet the elbow at $(b, n)$. Note that $p > p_{i-1}$ since $(c,p_{i-1})$ is not a \begin{tikzpicture}[scale=\inlinescale]
\draw (0, 0) grid (1,1);
\draw[blue](0,0.5)--(1,0.5);
\end{tikzpicture} tile and is occupied by pipe $a_{i-1}$. At $(c, p)$, there is pipe $c$ coming horizontally from the right and the tile $(c+1, p)$ right below is a  \begin{tikzpicture}[scale=\inlinescale]
\draw (0, 0) grid (1,1);
\draw[blue](0,0.5)--(1,0.5);
\draw[blue](0.5,0)--(0.5,1);
\end{tikzpicture}, so we must have an \begin{tikzpicture}[scale=\inlinescale]
            \draw (0, 0) grid (1,1);
            \draw[blue](1,0.5) arc (90:180:0.5);
        \end{tikzpicture} elbow since $(c,p)$ is not a \begin{tikzpicture}[scale=\inlinescale]
\draw (0, 0) grid (1,1);
\draw[blue](0,0.5)--(1,0.5);
\draw[blue](0.5,0)--(0.5,1);
\end{tikzpicture}. For $p < t < n$, the tile at $(c, t)$ must be a \begin{tikzpicture}[scale=\inlinescale]
\draw (0, 0) grid (1,1);
\draw[blue](0,0.5)--(1,0.5);
\draw[blue](0.5,0)--(0.5,1);
\end{tikzpicture} since there is pipe $c$ moving horizontally and there is a vertical pipe below. For $p_{i-1} < t < p$, at $(c, t)$ the green pipes below must keep going straight since the rightmost green pipe that turns with a \begin{tikzpicture}[scale=\inlinescale]
            \draw (0, 0) grid (1,1);
            \draw[blue](1,0.5) arc (90:180:0.5);
        \end{tikzpicture} elbow would run into the elbow at $(c, p)$. (Note that turning with \begin{tikzpicture}[scale=\inlinescale]
        \draw (0, 0) grid (1,1);
        \draw[blue](0,0.5) arc (90:0:0.5);
        \end{tikzpicture} elbow is not possible since the pipe would be moving right.)
        Now, let $\tau = \sigma t_{a_{i-1}a}$. Then, since pipe $a_{i-1}$ crosses pipe $t$ at $(t, p_{i-1})$ for $c < t < a_{i-1}$, we have \begin{equation}m = \sigma(a) = \tau(a_{i-1}) < \tau(t)  = \sigma(t)< p_{i-1} = \sigma(a_{i-1})\,.\end{equation}
        For all $a_{i-1} < t < a$, we have $\sigma(a_{i-1})> \sigma(t) > \sigma(a)$ by Remark \ref{cond} and we proved above that, for all $c < t < a_{i-1}$, we have $\sigma(a_{i-1}) > \sigma(t) > \sigma(a)$. Furthermore, we have $\sigma(c) > \sigma(a_{i-1}) > \sigma(a)$. Thus, $\sigma(c) > \sigma(t) > \sigma(a)$ for all $c < t < a$, so $c\in S$. Since $c < a_{i-1}$, then $c = a_j$ for some $j > i-1$. In particular, $a_i$ exists, implying that $i-1 < k$ (recall that $S = \{a_1 > a_2 > \dots > a_k\}$). As a result, the set of QBPDs of $\sigma t_{a_ka}$ with configuration $Y_k$ is empty. For $c <t < a_{i-1}$, since $\sigma(t) < \sigma(a_{i-1})$, we have $t\notin S$ and $c\in S$, so we have $c = a_{i}$. Thus, $p = \tau(a_{i}) = p_i$. Thus, we have the configuration as in the RHS.
        
        Now, if we start with a QBPD on the LHS and make the changes in Figure \ref{fig:phii}, the changed pipes on the RHS are all in the region (and thus can be seen to not intersect any other pipe twice) except the new $a_{i-1}$ pipe. This new $a_{i-1}$ pipe intersects with some vertical green pipes, but it will not intersect with them again since pipes do not move right, and then it intersects with pipe $t$ for $a_{i} < t < a_{i-1}$. After this, it takes on the path of the old $a_i$ pipe, so it remains to check that this old $a_i$ pipe does not intersect with pipe $t$ for $a_{i} < t < a_{i-1}$. Let $w = \sigma t_{a_ia}$. We have $w(a_i) = \sigma(a)  < \sigma(k)$ for all $a_i < t < a$ since $\sigma t_{a_ia}\lhd \sigma$, so we have that old pipe $a_i$ does not intersect pipe $t$ for $a_{i} < k < a_{i-1}$.
        
        If we start with a QBPD on the RHS and make the reverse changes, the only changed pipe that could possibly intersect another pipe twice is the new pipe $a_i$. However, this pipe first intersects some green vertical pipes and never intersects them again since pipes do not move right. After this, it takes on the path of the old $a_{i-1}$ pipe, which does not intersect any other pipe twice.  The RHS has a \begin{tikzpicture}[scale=\inlinescale]
\draw (0, 0) grid (1,1);
\draw[blue](0.5,0) arc (0:90:0.5);
\end{tikzpicture} or \begin{tikzpicture}[scale=\inlinescale]
\draw (0, 0) grid (1,1);
\draw[blue](0.5,0)--(0.5,1);
\end{tikzpicture} in which the pipe moves upward at $(a_i, p_{i-1})$, which contributes $-q_{a_i}$ and for $a_i < t <a_{i-1}$, there is a \begin{tikzpicture}[scale=\inlinescale]
\draw (0, 0) grid (1,1);
\draw[blue](0.5,0)--(0.5,1);
\draw[blue](0,0.5)--(1,0.5);
\end{tikzpicture} tile at $(t, p_{i-1})$ in which the vertical pipe moves upward which contributes $q_t$; these tiles are not present in the LHS. Thus, the weight is multiplied by $-q_{a_ia}$.
    \end{proof}
\end{lemma} 
Now we are ready to prove Theorem \ref{thm:qdform}.
\begin{proof}[Proof of Theorem \ref{thm:qdform}]
Let \begin{equation}T_w:= \sum_{P\in \QBPD(w)}\bwt(P)\end{equation}
By Proposition \ref{prop:stabform}, it suffices to show that $T_w$ satisfies \eqref{eq:transition}.
We have 
\begin{align}
   T_\pi &= \sum_{P\in \QBPD(\pi)}\bwt(P)\notag\\
    &= \sum_{P\in \QBPD(\pi, A)}\bwt(P)+\sum_{P\in \QBPD(\pi, B)}\bwt(P)+\sum_{P\in \QBPD(\pi, C)}\bwt(P)\notag\\
    &+\sum_{P\in \QBPD(\pi, D)}\bwt(P)+\sum_{i = 2}^k\sum_{P \in \QBPD(\sigma t_{a_ia}, X_i)}q_{a_{i}a}(\bwt(P)-\bwt(P))\,.
\end{align}
Applying Lemmas \ref{lem:phia}, \ref{lem:phib}, \ref{lem:phic}, \ref{lem:phid}, \ref{lem:phii} to the 5 summations, respectively, we get that
\begin{align}
    T_\pi&=(x_a-y_m) T_{\sigma} +\sum_{\substack{ c < a,\\ \sigma t_{ca}\gtrdot\sigma}} T_{\sigma t_{ca}}-\sum_{\substack{a < c,\\ \sigma t_{ac}\lhd\sigma}}q_{ac}T_{\sigma t_{ac}}+\sum_{i = 1}^k\sum_{P \in \QBPD(\sigma t_{a_ia})}q_{a_ia}\bwt(P)\notag\\
    &=(x_a-y_m) T_{\sigma} +\sum_{\substack{ c < a,\\ \sigma t_{ca}\gtrdot\sigma}} T_{\sigma t_{ca}}-\sum_{\substack{a < c,\\ \sigma t_{ac}\lhd\sigma}}q_{ac}T_{\sigma t_{ac}}+\sum_{\substack{c < a,\\ \sigma t_{ca}\lhd\sigma}}q_{ca}T_{\sigma t_{ca}},
\end{align}
as desired. 
\end{proof}

\section{Cancellation Analysis}\label{cancel}
There is no cancellation in the bumpless pipe dream formula for the double Schubert polynomials. Quantum bumpless pipe dreams provide a combinatorial formula for the monomial expansion of the quantum double Schubert polynomial, but this formula is not cancellation-free. In this section, we give some analysis of how much cancellation occurs and what kind of cancellation occurs for quantum double Schubert polynomials.

Table \ref{tab:s4cancel} lists all the permutations in $S_4$ for which the QBPD formula gives cancellation when computing the quantum double Schubert polynomials. In $S_4$, cancellations occur infrequently due to the limited number of QBPDs generated by permutations. Conversely, in $S_5$ and $S_6$, the frequency of cancellations rises alongside the number of distinct QBPDs, as cancellations manifest between these distinct QBPDs.
\begin{table}[h!]
    \centering
    \begin{tabular}{ c c c c c }
     Permutation &   Monomials & QBPD Monomials &  Cancellations & Number of QBPDs\\
     $[4,1,3,2]$ & $50$ & $54$ & $2$ & $9$\\
     $[3,1,4,2]$ & $18$ & $20$ & $1$ & $4$\\
     $[1,4,3,2]$ & $46$ & $48$ & $1$ & $9$\\
     $[2,1,4,3]$ & $12$ & $14$ & $1$ & $5$\\
\end{tabular}
    \caption{Nonzero cancellations for QBPDs in $S_4$, when considering the generated quantum double Schubert polynomial.}
    \label{tab:s4cancel}
\end{table}
\begin{example}
    Permutation $615432$ has $97032$ monomials in its quantum double Schubert polynomial, while the total number of monomials generated from QBPDs is $140052$. $21510$ pairs of monomials cancel out, and the number of QBPDs is $1038$. This is the permutation with the most cancellation in $S_6$.
\end{example}
\begin{table}[h!]
    \centering
    \begin{tabular}{ c c c c c c }
         & Total & Average per Permutation & Permutation of Max & Max\\
         $S_3$ & $0$ & $-$ & $-$ & $-$\\
         $S_4$ & $5$ & $0.208$ & $[4,1,3,2]$ & $2$\\ 
         $S_5$ & $1350$ & $11.25$ & $[5,1,4,3,2]$ & $153$\\ 
         $S_6$ & $570549$ & $792.43$ & $[6,1,5,4,3,2]$ & $21510$
    \end{tabular}
    \caption{Cancellations in $S_n$ for $n=3,4,5,6$}
    \label{tab:cancelstats}
\end{table}
As shown in Table~\ref{tab:cancelstats}, the number of cancellations per permutation grows larger with greater $n$. These cancellations increase with the number of QPBDs that are generated by a permutation. In particular, we can observe that there are several ways that cancellations occur. Two QBPDs could completely cancel each other out in both the single and double quantum Schubert polynomial case, as in Figure \ref{fig:completecancel}, or they could partially cancel. In the case of partial cancellation, Figure \ref{fig:binomial-cancel} illustrates the cancellation of binomial terms and no cancellation in monomial terms, while Figure \ref{fig:12543} illustrates the cancellation of monomial terms but not binomial terms. 
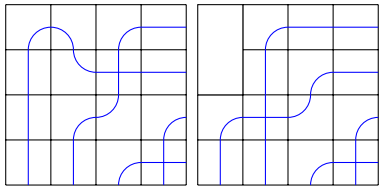
\begin{figure}[h!]
    \centering
    \begin{tikzpicture}[scale=\figscale]
\draw (0, 0) grid (4,4);
\draw[blue](1,3.5) arc (90:180:0.5);
\draw[blue](0.5,2)--(0.5,3);
\draw[blue](0.5,1)--(0.5,2);
\draw[blue](0.5,0)--(0.5,1);
\draw[blue](1.5,3) arc (0:90:0.5);
\draw[blue](1.5,3) arc (180:270:0.5);
\draw[blue](2,1.5) arc (90:180:0.5);
\draw[blue](1.5,0)--(1.5,1);
\draw[blue](3,3.5) arc (90:180:0.5);
\draw[blue](2.5,2)--(2.5,3);
\draw[blue](2,2.5)--(3,2.5);
\draw[blue](2,1.5) arc (-90:0:0.5);
\draw[blue](3,0.5) arc (90:180:0.5);
\draw[blue](3,3.5)--(4,3.5);
\draw[blue](3,2.5)--(4,2.5);
\draw[blue](4,1.5) arc (90:180:0.5);
\draw[blue](3.5,0)--(3.5,1);
\draw[blue](3,0.5)--(4,0.5);
\end{tikzpicture}
\begin{tikzpicture}[scale=\figscale]
\draw (0, 0) grid (4,4);
\draw[fill=white,draw=black](0,2) rectangle (1,4);
\draw[blue](1,1.5) arc (90:180:0.5);
\draw[blue](0.5,0)--(0.5,1);
\draw[blue](2,3.5) arc (90:180:0.5);
\draw[blue](1.5,2)--(1.5,3);
\draw[blue](1.5,1)--(1.5,2);
\draw[blue](1,1.5)--(2,1.5);
\draw[blue](1.5,0)--(1.5,1);
\draw[blue](2,3.5)--(3,3.5);
\draw[blue](3,2.5) arc (90:180:0.5);
\draw[blue](2,1.5) arc (-90:0:0.5);
\draw[blue](3,0.5) arc (90:180:0.5);
\draw[blue](3,3.5)--(4,3.5);
\draw[blue](3,2.5)--(4,2.5);
\draw[blue](4,1.5) arc (90:180:0.5);
\draw[blue](3.5,0)--(3.5,1);
\draw[blue](3,0.5)--(4,0.5);
\end{tikzpicture}
    \caption{Two QBPDs for $2143$ whose binomial weights cancel each other out completely. The left contributes $-q_1$, and the right contributes $q_1$.}
    \label{fig:completecancel}
\end{figure}
\begin{figure}[h!]
    \centering
    \begin{tikzpicture}[scale=\figscale]
\draw (0, 0) grid (4,4);
\draw[fill=white,draw=black](0,2) rectangle (1,4);
\draw[blue](1,1.5) arc (90:180:0.5);
\draw[blue](0.5,0)--(0.5,1);
\draw[blue](2,2.5) arc (90:180:0.5);
\draw[blue](1,1.5) arc (-90:0:0.5);
\draw[blue](2,0.5) arc (90:180:0.5);
\draw[blue](3,3.5) arc (90:180:0.5);
\draw[blue](2,2.5) arc (-90:0:0.5);
\draw[blue](3,1.5) arc (90:180:0.5);
\draw[blue](2.5,0)--(2.5,1);
\draw[blue](2,0.5)--(3,0.5);
\draw[blue](3,3.5)--(4,3.5);
\draw[blue](4,2.5) arc (90:180:0.5);
\draw[blue](3.5,1)--(3.5,2);
\draw[blue](3,1.5)--(4,1.5);
\draw[blue](3.5,0)--(3.5,1);
\draw[blue](3,0.5)--(4,0.5);
\end{tikzpicture}
\begin{tikzpicture}[scale=\figscale]
\draw (0, 0) grid (4,4);
\draw[blue](1,3.5) arc (90:180:0.5);
\draw[blue](0.5,2)--(0.5,3);
\draw[blue](0.5,1)--(0.5,2);
\draw[blue](0.5,0)--(0.5,1);
\draw[blue](1.5,3) arc (0:90:0.5);
\draw[blue](1.5,3) arc (180:270:0.5);
\draw[blue](2,0.5) arc (90:180:0.5);
\draw[blue](3,3.5) arc (90:180:0.5);
\draw[blue](2,2.5) arc (-90:0:0.5);
\draw[blue](3,1.5) arc (90:180:0.5);
\draw[blue](2.5,0)--(2.5,1);
\draw[blue](2,0.5)--(3,0.5);
\draw[blue](3,3.5)--(4,3.5);
\draw[blue](4,2.5) arc (90:180:0.5);
\draw[blue](3.5,1)--(3.5,2);
\draw[blue](3,1.5)--(4,1.5);
\draw[blue](3.5,0)--(3.5,1);
\draw[blue](3,0.5)--(4,0.5);
\end{tikzpicture}
    \caption{Two QBPDs for $1432$ whose monomial weights do not cancel out, but binomial weights partially cancel out. The left QBPD contributes $x_1q_1-y_2q_1$, and the right QBPD contributes $-x_3q_1+y_2q_1$.}
    \label{fig:binomial-cancel}
\end{figure}
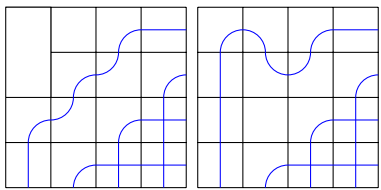
\begin{figure}[h!]
    \centering
    \begin{tikzpicture}[scale=\figscale]
\draw (0, 0) grid (5,5);
\draw[fill=white,draw=black](0,3) rectangle (1,5);
\draw[blue](1,2.5) arc (90:180:0.5);
\draw[blue](0.5,1)--(0.5,2);
\draw[blue](0.5,0)--(0.5,1);
\draw[blue](2,4.5) arc (90:180:0.5);
\draw[blue](1.5,3)--(1.5,4);
\draw[blue](1,2.5) arc (-90:0:0.5);
\draw[blue](2,1.5) arc (90:180:0.5);
\draw[blue](1.5,0)--(1.5,1);
\draw[blue](2,4.5)--(3,4.5);
\draw[blue](3,3.5) arc (90:180:0.5);
\draw[blue](2.5,2)--(2.5,3);
\draw[blue](2,1.5) arc (-90:0:0.5);
\draw[blue](3,0.5) arc (90:180:0.5);
\draw[blue](3,4.5)--(4,4.5);
\draw[blue](3,3.5)--(4,3.5);
\draw[blue](4,1.5) arc (90:180:0.5);
\draw[blue](3.5,0)--(3.5,1);
\draw[blue](3,0.5)--(4,0.5);
\draw[blue](4,4.5)--(5,4.5);
\draw[blue](4,3.5)--(5,3.5);
\draw[blue](5,2.5) arc (90:180:0.5);
\draw[blue](4.5,1)--(4.5,2);
\draw[blue](4,1.5)--(5,1.5);
\draw[blue](4.5,0)--(4.5,1);
\draw[blue](4,0.5)--(5,0.5);
\end{tikzpicture}
\begin{tikzpicture}[scale=\figscale]
\draw (0, 0) grid (5,5);
\draw[blue](1,4.5) arc (90:180:0.5);
\draw[blue](0.5,3)--(0.5,4);
\draw[blue](0.5,2)--(0.5,3);
\draw[blue](0.5,1)--(0.5,2);
\draw[blue](0.5,0)--(0.5,1);
\draw[blue](1.5,4) arc (0:90:0.5);
\draw[blue](1.5,4) arc (180:270:0.5);
\draw[blue](2,1.5) arc (90:180:0.5);
\draw[blue](1.5,0)--(1.5,1);
\draw[blue](3,4.5) arc (90:180:0.5);
\draw[blue](2,3.5) arc (-90:0:0.5);
\draw[blue](3,2.5) arc (90:180:0.5);
\draw[blue](2,1.5) arc (-90:0:0.5);
\draw[blue](3,0.5) arc (90:180:0.5);
\draw[blue](3,4.5)--(4,4.5);
\draw[blue](4,3.5) arc (90:180:0.5);
\draw[blue](3,2.5) arc (-90:0:0.5);
\draw[blue](4,1.5) arc (90:180:0.5);
\draw[blue](3.5,0)--(3.5,1);
\draw[blue](3,0.5)--(4,0.5);
\draw[blue](4,4.5)--(5,4.5);
\draw[blue](4,3.5)--(5,3.5);
\draw[blue](5,2.5) arc (90:180:0.5);
\draw[blue](4.5,1)--(4.5,2);
\draw[blue](4,1.5)--(5,1.5);
\draw[blue](4.5,0)--(4.5,1);
\draw[blue](4,0.5)--(5,0.5);
\end{tikzpicture}
    \caption{Two QBPDs for $12543$ whose monomial weights cancel each other out, but binomial weights do not cancel completely. The left QBPD contributes $x_3q_1-y_4q_1$, and the right contributes $-x_3q_1+y_2q_1$.}
    \label{fig:12543}
\end{figure}
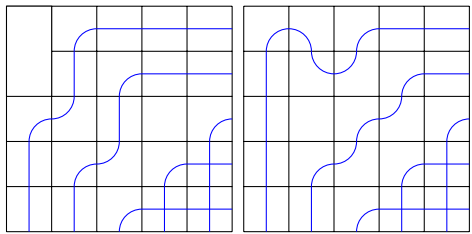
We attempted to find natural classes of permutations that are cancellation-free; however, we were unable to do so. Grassmannian, $132$-avoiding, dominant, and vexillary permutations all fail to be cancellation-free in general. However, we do observe that simple cycles of the form $(t, t+1, \dots, t+k)$ or of the form $(t, t+k, t+k-1, \dots, t+1)$ in cycle notation, as well as the longest permutation $w_0$, are cancellation free.
\begin{proposition}\label{prop:cancelfree}
    The QBPD formula for the quantum double Schubert polynomials is cancellation-free for the following classes of permutations: the longest permutation $w_0$ of each $S_n$, as well as cycles of the form $(t, t+1, \dots, t+k)$ and $(t, t+k, t+k-1, \dots, t+1)$ in cycle notation for $k \geq 1$.
    \begin{proof}
        For the longest permutation $w_0$ of each $S_n$, the Rothe diagram is the only unpaired QBPD. As a result, there are no tiles contributing a $-q_i$ factor to the weight, so there are no cancellations.

        The simple cycle of the form $(t,t+1,\dots, t+k)$ has a Rothe diagram with only one single column of empty tiles at column $t$. Any unpaired QBPDs can then be obtained using only droop moves. As a result, there are no tiles contributing a $-q_i$ factor to the weight, and therefore no cancellations. 

        The simple cycle of the form $(t,t+k,t+k-1,\dots,t+1)$ has a Rothe diagram with only one single row of empty tiles at row $t$. Any QBPDs can then be obtained using droop moves or lift moves that lift up the strand by at most one row. Thus, the only tiles that can contribute $q$ variables are \begin{tikzpicture}[scale=\inlinescale]
\draw (0, 0) grid (1,1);
\draw[blue](0.5,0) arc (0:90:0.5);
\end{tikzpicture} tiles, so there are also no cancellations.
    \end{proof}
\end{proposition}
\section{Future Directions}\label{future}
This paper provides a bumpless pipe dream formulation for quantum double Schubert polynomials, but we were unable to find a cancellation-free formula. Obtaining such a formula for quantum double Schubert polynomials would be desirable.

The authors also attempted to find a quantum (non-bumpless) pipe dream formulation for the quantum double Schubert polynomials, generalizing the usual pipe dream formulation of Schubert polynomials as in \cite{FOMIN1996123, bergeron}, but were not able to do so successfully. A canonical weight-preserving bijection between pipe dreams and bumpless pipe dreams was given in \cite{Gao_2023}; however, it only preserves the $x$'s weight. It would be interesting to come up with a pipe dream formulation for quantum double Schubert polynomials, and it would also be interesting if there is a canonical bijection from such objects to the QBPDs in this paper that preserves the monomial weights.
\bibliographystyle{plain}
\bibliography{Citations}

 \end{document}